\documentclass[review,11pt, english,times]{elsarticle}
\PassOptionsToPackage{normalem}{ulem}
\usepackage{ulem}
\usepackage{babel}
\usepackage{color}
\usepackage{amsmath}
\usepackage{amssymb}
\usepackage{amsthm}
\usepackage{float}
\usepackage{graphicx}
\usepackage{multirow}
\usepackage{epstopdf}
\usepackage{geometry}
\usepackage{subcaption}

\makeatletter
\numberwithin{equation}{section}
\numberwithin{figure}{section}
  \theoremstyle{remark}
  \newtheorem*{notation*}{\protect\notationname}
\theoremstyle{plain}
\newtheorem{thm}{\protect\theoremname}
  \theoremstyle{definition}
  
  \theoremstyle{plain}
  \newtheorem{lem}[thm]{\protect\lemmaname}
  \theoremstyle{remark}
  \newtheorem{rem}[thm]{\protect\remarkname}
  \theoremstyle{definition}
\newtheorem{exm}{Example}[section]

\makeatother

  \providecommand{\definitionname}{Definition}
  \providecommand{\lemmaname}{Lemma}
  \providecommand{\notationname}{Notation}
  \providecommand{\remarkname}{Remark}
\providecommand{\theoremname}{Theorem}

\begin{document}

\begin{frontmatter}

\title{The Cauchy problem of coupled elliptic sine-Gordon equations with noise: 
Analysis of a general kernel-based regularization and reliable tools of computing}


\author[aff1]{Vo Anh Khoa}
\ead{khoa.vo@gssi.infn.it, vakhoa.hcmus@gmail.com}

\author[aff2]{Mai Thanh Nhat Truong}
\ead{mtntruong@toantin.org}

\author[aff3]{Nguyen Ho Minh Duy}
\ead{nhmduy.hcmus@gmail.com}

\author[aff4]{Nguyen Huy Tuan\corref{mycorrespondingauthor}}
\cortext[mycorrespondingauthor]{Corresponding author}
\ead{nguyenhuytuan@tdt.edu.vn}

\address[aff1]{Mathematics and Computer Science Division,
Gran Sasso Science Institute, L'Aquila, Italy}
\address[aff2]{Department of Electrical, Electronic and Control Engineering, Hankyong National University, Anseong, Gyeonggi, Republic of Korea}
\address[aff3]{Department of Mathematics and Computer Science, Ho Chi Minh City
University of Science, Ho Chi Minh City, Vietnam}
\address[aff4]{Applied Analysis Research Group, Faculty of Mathematics and Statistics, Ton Duc Thang University, Ho Chi Minh City, Vietnam}

\begin{abstract}
Developments in numerical methods for problems governed by nonlinear
partial differential equations underpin simulations with sound arguments in diverse areas 
of science and engineering. In this paper, we explore the regularization method for the coupled
elliptic sine-Gordon equations along with Cauchy data. The system of equations originates
from the static case of the coupled hyperbolic sine-Gordon equations
modeling the coupled Josephson junctions in superconductivity, and
so far it addresses the Josephson $\pi$-junctions. In general, the Cauchy problem
is not well-posed, and herein the Hadamard-instability occurs drastically. 
Generalizing the kernel-based regularization method, we propose a stable approximate solution. 
Confirmed by the error estimate, this solution strongly converges to the exact solution in 
$L^2$-norm. The main concern of this paper is also with the way to compute the regularized solution formed by 
an alike integral equation. We employ the proposed techniques that successfully approximated the highly 
oscillatory integral, and apply the Picard-like iteration to organize an efficient and reliable tool of computations.
The results are viewed as the improvement as well as the generalization 
of many previous works. The paper is  also accompanied 
by a numerical example that demonstrates the potential of this idea.
\end{abstract}

\begin{keyword}
Elliptic sine-Gordon equations\sep Ill-posedness\sep Regularization methods\sep
Stability\sep Convergence rate \sep Highly oscillatory integrals \sep Iteration
\MSC[2010] 47A52\sep 20M17\sep 26D15\sep 65D30\sep 65D15
\end{keyword}

\end{frontmatter}

\section{Introduction}

The Josephson junction is a quantum mechanical device which is formed
by two superconducting electrodes separated by a very thin insulating
barrier. The Josephson $\pi$-junction is a specific example of that
when no external current or magnetic field is applied. It is the corner
junction made of yttrium barium copper oxide ($d$-wave high-temperature
superconductivity). The fundamental equations modeling such a junction
(see Chen et al. \cite{Chen03}) read
\[
\frac{\partial u}{\partial t}=\frac{2EV}{P},\quad\frac{\partial u}{\partial x}=\left(\frac{2Ed}{P\omega}\right)H_{2},\quad\frac{\partial u}{\partial y}=\left(-\frac{2Ed}{P\omega}\right)H_{1},
\]
\[
J\left(x,y\right)=-J_{0}\left(x,y\right)\sin u,
\]
where $u$ is functional to describe the relative phase between the
superconducting metals $I$ and $II$, causing the Josephson tunneling
current $J$ per unit area which depends on the properties $J_{0}$
of the barrier. Experimentally, $E$ corresponds to the electron charge,
$V$ is considered as a time-and-space-dependent potential difference
across the barrier, $P$ is the Planck constant, $d$ is the constant
with respect to the London penetration depths for the metals, $\omega$
performs the speed of light, and $H_{j}$ represents the $x$ and
$y$ component of the magnetic field.

After substituting those equations into the Maxwell equation, it yields
the barrier equation or the hyperbolic sine-Gordon equation
\begin{equation}
\frac{\partial^{2}u}{\partial x^{2}}+\frac{\partial^{2}u}{\partial y^{2}}-c_{0}\frac{\partial^{2}u}{\partial t^{2}}=-\gamma^{-2}\sin u,\label{eq:1.1}
\end{equation}
where $c_{0}$ depends only on $\omega$ and $\gamma$ includes $\omega,P,E,d$
and $J_{0}$.

Remarkably, the static version of (\ref{eq:1.1}) reads
\[
\frac{\partial^{2}u}{\partial x^{2}}+\frac{\partial^{2}u}{\partial y^{2}}=-\gamma^{-2}\sin u,
\]
and if we proceed the change of variable $v\left(x,y\right)=\pi-u\left(x,y\right)$
under the homogeneous Neumann boundary condition, we are thus concerned
with the equation
\[
\frac{\partial^{2}u}{\partial x^{2}}+\frac{\partial^{2}u}{\partial y^{2}}=\gamma^{-2}\sin u,
\]
which addresses us to the relation of the Josephson $\pi$-junction.

Mathematically speaking, the Neumann condition guarantees the equivalence
between those two kinds of junctions, but the others, for example,
the homogeneous Dirichlet boundary condition, are not really ascertained.

{The system of damped and coupled sine-Gordon equations considered as the dynamical model for the coupled Josephson junctions in \cite{Le84} has the following form:
\[
\frac{\partial^{2}u_{1}}{\partial t^{2}}+\frac{\partial u_{1}}{\partial t}-\Delta u_{1}+\sin u_{1}+k\left(y_{1}-y_{2}\right)=f_{1},
\]
\[
\frac{\partial^{2}u_{2}}{\partial t^{2}}+\frac{\partial u_{2}}{\partial t}-\Delta u_{2}+\sin u_{2}-k\left(y_{1}-y_{2}\right)=f_{2}.
\]

On the other hand, considering the system of coupled Klein-Gordon equations
\[
\frac{\partial^{2}u_{1}}{\partial t^{2}}-\Delta u_{1}=f_{u_{1}}\left(u_{1},u_{2}\right),\quad
\frac{\partial^{2}u_{2}}{\partial t^{2}}-c^2\Delta u_{2}=f_{u_{2}}\left(u_{1},u_{2}\right),
\]
yields another type of coupled sine-Gordon equations 
\[
\frac{\partial^{2}u_{1}}{\partial t^{2}}-\Delta u_{1}=-\delta^{2}\sin\left(u_{1}-u_{2}\right),
\quad
\frac{\partial^{2}u_{2}}{\partial t^{2}}-c^{2}\Delta u_{2}=-\sin\left(u_{1}-u_{2}\right),
\]
if we choose the potential function $f\left(u_{1},u_{2}\right)=\cos\left(\delta u_{1}-u_{2}\right)-1$.

From that motivation, the purpose of this paper is to consider the following
system in two-dimensional:
\begin{equation}
\frac{\partial^{2}u_{1}}{\partial x^{2}}+\alpha_{1}\frac{\partial^{2}u_{1}}{\partial y^{2}}+\gamma_{1}\sin\left(\delta_{11}u_{1}+\delta_{12}u_{2}\right)+\sigma_{11}u_{1}+\sigma_{12}u_{2}=f_{1},\label{eq:1.2}
\end{equation}

\begin{equation}
\frac{\partial^{2}u_{2}}{\partial x^{2}}+\alpha_{2}\frac{\partial^{2}u_{2}}{\partial y^{2}}+\gamma_{2}\sin\left(\delta_{21}u_{1}+\delta_{22}u_{2}\right)+\sigma_{21}u_{1}+\sigma_{22}u_{2}=f_{2},\label{eq:1.3}
\end{equation}
where $\alpha_{i}\in\mathbb{R}_{+},\gamma_{i},\delta_{ij},\sigma_{ij}\in\mathbb{R}$
are physical constants and $f_{i}$ are called forcing functions,
$i,j\in\left\{ 1,2\right\} $.

It is worth noting that in the context of static and long Josephson junctions, $u_{1}$ and $u_{2}$ represent the distribution of two superconducting phase differences between the quantum-mechanical wave functions of four Josephson junctions. We herein aim at determining the distribution on the surface in case we only have the interior measurements.}

There are hundreds of studies on the hyperbolic version of (\ref{eq:1.2})-(\ref{eq:1.3}).
For instance, Levi et al. \cite{Le84} considered the chaotic behaviors
of numerical solutions for the coupled hyperbolic sine-Gordon equations
with damped terms. The more general system can be regarded as the
problem of identification the physical constants by Ha and Nakagiri
in \cite{HN02} and the answer for the question of necessary condition
to gain optimality for those parameters can be found there. In some applications we see, the system of coupled hyperbolic sine-Gordon equations is viewed as the generalization of the
Frenkel-Kontorova dislocation model in \cite{BK98}. On the other side, it can be easily and transparently extended to the soliton excitation in DNA double helices by Yomosa \cite{Yo83}. Up-to-date,
the amount of papers which are related to those problems still increases
without ceasing. We, however, stress that the results as far as we know
for the elliptic sine-Gordon equations are very scarce. In particular,
we only find some theoretical and numerical researches, e.g. in \cite{Chen03,Cap98,NS97,Sha92,Ray06}
and references therein. Additionally, it is worth noting that in principle 
the Cauchy problem for such equations is ill-posed where the stability of solution fails intrinsically. Consequently,
our analysis will underscore the so-called regularization method.

In recent years, various types of regularization methods have been
of continuous interest to researchers in a wide range of disciplines.
We refer the reader to many interesting results in \cite{ARRV09,BD10,Bour07,HDL09,HL00,RT09}.
We notably remark the studies of Tuan and his co-authors on the nonlinear
elliptic equations for which they are strongly relative to our system (\ref{eq:1.2})-(\ref{eq:1.3}) being tackled. In fact, the authors considered
in \cite{TTTK15} the modified method to show that there exists a
filtering kernel stabilizing the exponential instability when the
spectral representation of solution is used. Only a short time ago,
taking a more general nonlinear source term into account, they successfully
obtained in \cite{TTKT15} a new landmark by using the truncation
approach. In light of the aforementioned works, we put ourselves in
the study of the modified integral method to regularize the system (\ref{eq:1.2})-(\ref{eq:1.3}).

Let $\left(0,a\right)\times\left(0,b\right)$ for $a,b>0$, a couple
of real unknown functions $\left(u_{1},u_{2}\right)$ is sought for $\left(x,y\right)\in\left[0,a\right]\times\left[0,b\right]$.
The problem given by (\ref{eq:1.2})-(\ref{eq:1.3}) naturally along
with the no-flux boundary conditions:
\begin{equation}
\frac{\partial u_{i}}{\partial y}\left(x,0\right)=\frac{\partial u_{i}}{\partial y}\left(x,b\right)=0,\quad x\in\left(0,a\right),i\in\left\{1,2\right\},\label{eq:bound}
\end{equation}
and associated with initial conditions
\begin{equation}
u_{i}\left(0,y\right)=u_{0}^{i}\left(y\right),\quad\frac{\partial u_{i}}{\partial x}\left(0,y\right)=u_{1}^{i}\left(y\right),\quad y\in\left(0,b\right),i\in\left\{1,2\right\}.\label{eq:initial}
\end{equation}

In this paper, we consider problem (\ref{eq:1.2})-(\ref{eq:initial}),
and prove stability and convergence results of approximate solutions
constructed by the generalization of the modified integral method. Interestingly, this method is based on 
Fourier-mode in which the highly oscillatory integrals take part. It then motivates us to the use of 
the asymptotic expansions and the Filon-type methods which were introduced in \cite{Olv08,IN04,IN05}. 
Employing such methods and using the famous Picard-like iteration, we analyze 
and design a computational tool to compute the regularized solution. Until now, a rigorous focus on numerical procedures has not been made yet whilst the theoretical analysis grew significantly with a great amount of works. As a consequence, this current paper includes a novelty that makes a huge step in 
the area being considered.

Our paper is shortly organized as follows. Section \ref{sec:2} introduces some preliminaries and
notation, and show the ill-posedness in combination with proposing the regularized solution based on a filtering function. 
We also prepare some auxiliary results in this section. Section \ref{sec:main} is devoted to our main results including 
the well-posedness and convergence of the regularized solution. In Section \ref{sec:compute}, we carefully design a computational tool in the setting of finite-difference methods, and some constraints have been provided.  A numerical example is given in Section \ref{sec:test} to confirm our 
theoretical expectations. A short conclusion in Section \ref{sec:conclusion} closes the paper.

\section{Problem settings and notation}\label{sec:2}

\subsection{Abstract settings}

Let $\left\langle \cdot,\cdot\right\rangle $ be either the scalar
product in $L^{2}$ or the dual pairing of a continuous linear functional
and an element of a functional space. The notation $\left\Vert \cdot\right\Vert _{X}$
stands for the norm in the Banach space $X$. 

In $\mathcal{H}:=L^2(0,b)\times L^2(0,b)$, we define the usual inner product
\[
\left\langle U,V\right\rangle _{\mathcal{H}}:=\left\langle u_{1},v_{1}\right\rangle +\left\langle u_{2},v_{2}\right\rangle ,\quad\forall u_{i},v_{i}\in L^{2}\left(0,b\right),i\in\left\{ 1,2\right\} ,
\]
where $U=\left(u_1,u_2\right)^{t}$ and $V=\left(v_1,v_2\right)^{t}$ and the superscript $\cdot^{t}$ herein denotes the transposed operation, together with the norm
\[
\left\Vert U\right\Vert _{\mathcal{H}}^{2}:=\left\Vert u_{1}\right\Vert _{L^{2}\left(0,b\right)}^{2}+\left\Vert u_{2}\right\Vert _{L^{2}\left(0,b\right)}^{2}.
\]

In $X:=C([0,a];\mathcal{H})$, we define the norm
\[
\left\Vert U\right\Vert _{X}=\sup_{0\le x\le a}\left\Vert U\left(x,\cdot\right)\right\Vert _{\mathcal{H}},
\]

Let us also define the function space
\[
\mathcal{V}_{0}:=\left\{ U\in H^{1}\left(0,b\right)\times H^{1}\left(0,b\right):U_{y}\left(0\right)=U_{y}\left(b\right)=0\right\},
\]
the closed subspace of $\mathcal{V}:=H^{1}\left(0,b\right)\times H^{1}\left(0,b\right)$.

It is significant to remark that the positive-definite, bounded, symmetric operator
\[
\mathcal{A}:=\begin{pmatrix}-\Delta_{y} & 0\\
0 & -\Delta_{y}
\end{pmatrix},\quad-\Delta_{y}=-\frac{\partial^{2}}{\partial y^{2}},
\]
is an isomorphism from $\mathcal{V}_{0}$ to $\mathcal{V}$, and it is self-adjoint in $\mathcal{H}$. Therefore, $\mathcal{A}$ admits an orthonormal eigenbasis $\left\{ \phi_{n}\right\} _{n\in\mathbb{N}}$ in $\mathcal{H}$, associated with the eigenvalue
$\left\{ \lambda_{n}\right\} _{n\in\mathbb{N}}$ satisfying
\[
0<\lambda_{1}<\lambda_{2}<...<\lambda_{n}<...\lim_{n\to\infty}\lambda_{n}=\infty.
\]

Note that $\phi_{n}=\left(\phi_{n}^{1},\phi_{n}^{2}\right)^{t}$ and $\lambda_{n}=\left(\lambda_{n}^{1},\lambda_{n}^{2}\right)^{t}$ whilst the relation "$<$" above indicates $\lambda_{n}^{1}<\lambda_{n+1}^{1}$ and $\lambda_{n}^{2}<\lambda_{n+1}^{2}$. To this end, if not otherwise stated, other relations, e.g. "$\ge$", can be explained in the same meaning.  

We now turn to introduce the abstract Gevrey class of functions of order
$s=\left(s_1,s_2\right)^{t}$ and index $\nu=\left(\nu_1,\nu_2\right)^{t}$ whose entries are all positive, denoted by $\mathbb{G}_{\nu}^{s}$, defined
by
\[
\mathbb{G}_{\nu}^{s}:=\left\{ U\in \mathcal{H}:\sum_{n=1}^{\infty}\left(\left|\lambda_{n}^{1}\right|^{s_1/2}e^{2\nu_{1}\sqrt{\lambda_{n}^{1}}}\left|\left\langle u_{1},\phi_{n}^{1}\right\rangle \right|^{2}+\left|\lambda_{n}^{2}\right|^{s_2/2}e^{2\nu_{2}\sqrt{\lambda_{n}^{2}}}\left|\left\langle u_{2},\phi_{n}^{2}\right\rangle \right|^{2}\right)<\infty\right\} ,
\]
equipped with the norm
\[
\left\Vert U\right\Vert _{\mathbb{G}_{\nu}^{s}}^{2}=\sum_{n=1}^{\infty}\left(\left|\lambda_{n}^{1}\right|^{s_1/2}e^{2\nu_{1}\sqrt{\lambda_{n}^{1}}}\left|\left\langle u_{1},\phi_{n}^{1}\right\rangle \right|^{2}+\left|\lambda_{n}^{2}\right|^{s_2/2}e^{2\nu_{2}\sqrt{\lambda_{n}^{2}}}\left|\left\langle u_{2},\phi_{n}^{2}\right\rangle \right|^{2}\right)<\infty.
\]

This type of spaces forms as a scale of Hilbert spaces with respect to the index $s$ and the inner product is defined by
\[
\left\langle U,V\right\rangle_{\mathbb{G}_{\nu}^{s}}:=\left\langle (-\Delta_{y})^{s_{1}/2}e^{\nu_{1}(-\Delta_{y})^{1/2}}u_{1},(-\Delta_{y})^{s_{1}/2}e^{\nu_{1}(-\Delta_{y})^{1/2}}v_{1}\right\rangle
+\left\langle (-\Delta_{y})^{s_{2}/2}e^{\nu_{2}(-\Delta_{y})^{1/2}}u_{2},(-\Delta_{y})^{s_{2}/2}e^{\nu_{2}(-\Delta_{y})^{1/2}}v_{2}\right\rangle.
\]

To fulfill this subsection, we make the following assumptions:

$\left(\mbox{A}_{1}\right)$ For $i\in \left\{1,2\right\}$, assume $u_{0}^{i},u_{1}^{i}\in L^{2}\left(0,b\right)$ 
and the measured data $u_{0}^{i,\varepsilon},u_{1}^{i,\varepsilon}\in L^{2}\left(0,b\right)$. With the noise level $\varepsilon>0$, the difference between these exact and measured data can be written in $\mathcal{H}$-norm:
\[
\left\Vert U_{0}-U_{0}^{\varepsilon}\right\Vert_{\mathcal{H}} \le\varepsilon,
\quad\left\Vert U_{1}-U_{1}^{\varepsilon}\right\Vert_{\mathcal{H}} \le\varepsilon,
\]
where these notations follow $U$ mentioned above;

$\left(\mbox{A}_{2}\right)$ Suppose that both forcing terms $f_{1},f_{2}$ belong to $C\left(\left[0,a\right];L^{2}\left(0,b\right)\right)$.

Our objective in this paper is to develop the computational foundations,
we further assume that the system (\ref{eq:1.2})-(\ref{eq:initial})
has a unique solution in $\Omega$. In the next subsections, we review briefly 
the mild solution and then the Hadamard-instability is illuminated. Consequently, the modified method is designed together with some results needed to prove our main results.

\subsection{Ill-posedness and the kernel-based regularization method}

As is known, since we have the set of eigenbasis $\left\{\phi_{n}\right\}_{n\in\mathbb{N}}$, it is capable to present the solution of problem (\ref{eq:1.2})-(\ref{eq:initial}) by Fourier-mode, i.e.,
\[
u_{i}\left(x\right)=\sum_{n=1}^{\infty}\left\langle u_{i}\left(x\right),\phi_{n}\right\rangle \phi_{n},\quad i\in\left\{1,2\right\}.
\]

Under the nonlinear spectral theory, the functions $u_{i}$ are called the \emph{mild} solutions of problem (\ref{eq:1.2})-(\ref{eq:initial}) if they belong to $C\left(\left[0,a\right];L^{2}\left(0,b\right)\right)$ and satisfy the integral equations
\begin{eqnarray}
u_{i}\left(x\right) & = & \sum_{n=1}^{\infty}\left[\cosh\left(\sqrt{\alpha_{i}\lambda_{n}^{i}}x\right)\left\langle u_{0}^{i},\phi_{n}^{i}\right\rangle +\frac{\sinh\left(\sqrt{\alpha_{i}\lambda_{n}^{i}}x\right)}{\sqrt{\alpha_{i}\lambda_{n}^{i}}}\left\langle u_{1}^{i},\phi_{n}^{i}\right\rangle \right.\nonumber\\
 &  & \left.+\int_{0}^{x}\frac{\sinh\left(\sqrt{\alpha_{i}\lambda_{n}^{i}}\left(x-\xi\right)\right)}{\sqrt{\alpha_{i}\lambda_{n}^{i}}}\left\langle \mathcal{F}_{i}\left(\xi,u_{1},u_{2}\right),\phi_{n}^{i}\right\rangle d\xi\right]\phi_{n}^{i},\quad i\in\left\{ 1,2\right\}.\label{eq:umild}
\end{eqnarray}
where the function $\mathcal{F}_{i}$ is
given by
\begin{equation}
\mathcal{F}_{i}\left(\xi,u_1,u_2\right)=f_{i}(\xi)-\gamma_{i}\sin\left(\delta_{i1}u_{1}(\xi)+\delta_{i2}u_{2}(\xi)\right)-\sigma_{i1}u_{1}(\xi)-\sigma_{i2}u_{2}(\xi),\quad i\in\left\{1,2\right\}.\label{eq:F1}
\end{equation}

Therefore, let us observe \eqref{eq:umild}
that when $n\to\infty$, the rapid escalation of $\cosh\left(x\right)$
and $\frac{\sinh\left(x\right)}{x}$ tells us the ill-posedness of
the underlying problem in $L^{2}$. A small perturbation in the data,
 described by the assumption $\left(\mbox{A}_{1}\right)$
can arbitrarily show a huge error in computing the solution. Thus, performing
any computation is impossible in this case and then a regularization
method is required.

Now given $\beta=\beta\left(\varepsilon\right)\in (0,1)$, let us define the filtering
function $\Psi_{n,k}^{i,\beta}\left(\beta,x\right):\left[0,a\right]\to\mathbb{R}$ by
\begin{equation}
\Psi_{n,k}^{i,\beta}\left(x\right):=\frac{e^{-\sqrt{\alpha_{i}\lambda_{n}^{i}}\left(a-x\right)}}{2\beta\sqrt{\alpha_{i}^{k}\left|\lambda_{n}^{i}\right|^{k}}+2e^{-\sqrt{\alpha_{i}\lambda_{n}^{i}}a}},\quad n\in\mathbb{N},k\ge1,i\in\left\{ 1,2\right\} .\label{eq:filter}
\end{equation}

This function plays a major role in stabilizing the aforementioned unstable terms since it turns the unstable solution \eqref{eq:umild} into the stable solution in which one can prove well-posedness as well as the strong convergence to the exact solution. It has not escaped our notice that we are proposing the stable approximate solution associated with the measured data. Therefore, the so-called regularized solutions herein are given by
\begin{eqnarray}
u_{i}^{\varepsilon}\left(x\right) & = & \sum_{n=1}^{\infty}\left[\left(\Psi_{n,k}^{i,\beta}\left(x\right)+\frac{e^{-\sqrt{\alpha_{i}\lambda_{n}^{i}}x}}{2}\right)\left\langle u_{0}^{i,\varepsilon},\phi_{n}^{i}\right\rangle +\frac{1}{\sqrt{\alpha_{i}\lambda_{n}^{i}}}\left(\Psi_{n,k}^{i,\beta}\left(x\right)-\frac{e^{-\sqrt{\alpha_{i}\lambda_{n}^{i}}x}}{2}\right)\left\langle u_{1}^{i,\varepsilon},\phi_{n}^{i}\right\rangle \right.\nonumber \\
 &  & \left. +\int_{0}^{x}\frac{1}{\sqrt{\alpha_{i}\lambda_{n}^{i}}}\left(\Psi_{n,k}^{i,\beta}\left(x-\xi\right)-\frac{e^{-\sqrt{\alpha_{i}\lambda_{n}^{i}}\left(x-\xi\right)}}{2}\right)\left\langle \mathcal{F}_{i}\left(\xi,u_{1}^{\varepsilon},u_{2}^{\varepsilon}\right),\phi_{n}^{i}\right\rangle d\xi\right]\phi_{n}^{i},\quad i \in\left\{1,2\right\}.\label{eq:uepmild}
\end{eqnarray}

From here on, we shall move to the last part of this section where some fundamental results are provided to foster our estimates in Section \ref{sec:main}.

\subsection{Auxiliary results}

\begin{lem}\label{lem:boundedness}
(Boundedness of the kernel)
Given $\beta>0,k\ge1$ such that $a^{k}>k\beta$ then for $x\in\left[0,a\right]$
the following inequality holds:
\begin{equation}
\Psi_{n,k}^{i,\beta}\left(x\right)\le\frac{1}{2}\left(ka\right)^{\frac{kx}{a}}\beta^{-\frac{x}{a}}\left(\ln\left(\frac{a^{k}}{k\beta}\right)\right)^{-\frac{kx}{a}},\quad\forall n\in\mathbb{N},i\in\left\{ 1,2\right\} .\label{eq:2.8-1}
\end{equation}
\end{lem}
\begin{proof}
By taking the first derivatives of the function $z\mapsto \left(\beta z^k + e^{-Rz}\right)^{-1}$, we recognize that this function attains its maximum at $z=z_0$ where $z_{0}$
solves the equation $e^{-Rz_{0}}=\frac{k\beta}{R}z_{0}^{k-1}$. At this point we see, it implies that
\begin{equation}
\left(\beta z^k + e^{-Rz}\right)^{-1}\le \left(\beta z_{0}^{k}+\frac{k\beta}{R}z_{0}^{k-1} \right)^{-1},\quad z>0.\label{eq:2.8}
\end{equation}

On the other side, the elementary inequality $e^{Rz_{0}}\ge Rz_{0}$ yields that
\[
\frac{R}{k\beta}=z_{0}^{k-1}e^{Rz_{0}}\le\frac{1}{R^{k-1}}e^{kRz_{0}},
\]
which leads to
\begin{equation}
z_{0}\ge\frac{1}{kR}\ln\left(\frac{R}{k\beta}^{k}\right).\label{z000}
\end{equation}

Combining \eqref{z000} with \eqref{eq:2.8}, we arrive at
\[
\left(\beta z^k + e^{-Rz}\right)^{-1}\le\frac{1}{\beta z_{0}^{k}}\le\frac{1}{\beta}\left(\frac{kR}{\ln\left(\frac{R^{k}}{k\beta}\right)}\right)^{k}.\label{gggg}
\]

Using the expression
\[
\frac{e^{-rz}}{\beta z^k + e^{-Rz}}=\frac{e^{-rz}}{\left(\beta z^{k}+e^{-Rz}\right)^{\frac{r}{R}}\left(\beta z^{k}+e^{-Rz}\right)^{1-\frac{r}{R}}},
\]
then with \eqref{gggg}, we are led to the following estimate: for $0\le r\le R$
\begin{equation}
\frac{e^{-rz}}{\beta z^k + e^{-Rz}}\le\left(\beta z^k + e^{-Rz}\right)^{-1+\frac{r}{R}}\le\left(kR\right)^{k\left(1-\frac{r}{R}\right)}\beta^{\frac{r}{R}-1}\left(\ln\left(\frac{R^{k}}{k\beta}\right)\right)^{k\left(\frac{r}{R}-1\right)}.\label{eq:2.9}
\end{equation}

Henceforward, the proof of the lemma is now straightforward due to replacing
$z=\sqrt{\alpha_{i}\lambda_{n}^{i}},r=a-x$ and $R=a$ in (\ref{eq:2.9}).\end{proof}
\begin{rem}\label{rem:bou}
In (\ref{eq:2.9}), if we take $r=a-x+\xi$ for $x\ge\xi$ then another estimate holds:
\begin{equation}
\Psi_{n,k}^{i,\beta}\left(x-\xi\right)=\frac{e^{-\sqrt{\alpha_{i}\lambda_{n}^{i}}\left(a-x+\xi\right)}}{2\beta\sqrt{\alpha_{i}^{k}\left|\lambda_{n}^{i}\right|^{k}}+2e^{-\sqrt{\alpha_{i}\lambda_{n}^{i}}a}}
\le\frac{1}{2}\left(ka\right)^{\frac{k\left(x-\xi\right)}{a}}\beta^{\frac{\xi-x}{a}}\left(\ln\left(\frac{a^{k}}{k\beta}\right)\right)^{\frac{k\left(\xi-x\right)}{a}},\quad i\in\left\{ 1,2\right\} .\label{eq:2.10}
\end{equation}

\end{rem}
\begin{lem}\label{lem:lipschitzf}
The function $\mathcal{F}_{i}$, $i\in\left\{1,2\right\}$ in \eqref{eq:F1} is globally Lipschitz in the sense that
\begin{equation}
\left|\mathcal{F}_{i}\left(\xi,u_{1},u_{2}\right)-\mathcal{F}_{i}\left(\xi,v_{1},v_{2}\right)\right|\le\left(\left|\gamma_{i}\right|\left|\delta_{i1}\right|+\left|\sigma_{i1}\right|\right)\left|u_{1}-v_{1}\right|+\left(\left|\gamma_{i}\right|\left|\delta_{i2}\right|+\left|\sigma_{i2}\right|\right)\left|u_{2}-v_{2}\right|,\label{eq:2.17}
\end{equation}
\end{lem}
\begin{proof}
	The proof of the lemma is clear. Indeed, we follow the structure of the forcing functions to get
	\[
	\left|\mathcal{F}_{i}\left(\xi,u_{1},u_{2}\right)-\mathcal{F}_{i}\left(\xi,v_{1},v_{2}\right)\right|\le
	\gamma\left|\sin\left(\delta_{i1}u_{1}+\delta_{i2}u_{2}\right) -\sin\left(\delta_{i1}v_{1}+\delta_{i2}v_{2}\right)\right|.
	\]
	
	We then apply the elementary inequality $\left|\sin U-\sin V\right|\le \left|U-V\right|$ for $U,V\in\mathbb{R}$ and thus simplify the resulting estimate to gain the desired structural inequality.  
\end{proof}

\section{Main results}
\label{sec:main}
In this section, we explore two main points: the qualitative analysis of the modified solutions \eqref{eq:uepmild} and the convergence rate which tells us how fast one can approximate a pair of solutions $U=\left(u_1,u_2\right)^{t}$, the solutions of problem \eqref{eq:1.2}-\eqref{eq:1.3}. Normally, the notion of the regularization method is to approximate the ill-posed problem by using an orthoprojection in $L^2$ onto the subspace spanned by the set $\left\{\phi_{n}\right\}_{n\in \mathbb{N}}$, then the new problem is well-posed. Some works that follow such strategies can be referred to \cite{TTK15,TTKT15,TTTK15}. Likewise, we commence the first part by proving the existence and uniqueness of \eqref{eq:uepmild} based on the fixed-point argument. In parallel to this work, we lump together several estimates to gain the stability analysis. In the second part, we confirm the strong convergence in $L^2$ by investigating the error estimate between \eqref{eq:uepmild} and \eqref{eq:umild}.
\subsection{Well-posedness}
We notably mention that (\ref{eq:uepmild}) expresses precisely the integral equation $u^{\varepsilon}\left(x\right)=\mathcal{G}\left(u^{\varepsilon}\right)\left(x\right)$
where $\mathcal{G}$ defined by the right-hand side of (\ref{eq:uepmild})
maps from $X=C\left(\left[0,a\right];\mathcal{H}\right)$
into itself. At this stage, we may see the boundedness of the filtering function in (\ref{eq:2.8-1})
and (\ref{eq:2.10}) and the global Lipschitz in \eqref{eq:2.17} of the
functional $\mathcal{F}_{i}$ are very handy.

\begin{thm}\label{thm:unique}
(Existence and uniqueness)
The system of integral equations (\ref{eq:uepmild}) has a unique solution $U^{\varepsilon}\in C([0,a];\mathcal{H})$ for each $\varepsilon>0$.
\end{thm}
\begin{proof}
As mentioned before, the proof follows the well-known Banach fixed-point argument. At first, we rewrite the system \eqref{eq:uepmild} in the vectorial case in which $U^{\varepsilon} = (u_{1}^{\varepsilon},u_{2}^{\varepsilon})^{t}$ obeys the Volterra-type integral equation:
\begin{equation}
U^{\varepsilon}(x) = H^{\varepsilon}(x) + \int_{0}^{x}K^{\varepsilon}(x,\xi,U^{\varepsilon}(\xi))d\xi,
\end{equation}
where
\[
H^{\varepsilon}\left(x\right)=\left(\begin{array}{c}
\displaystyle{\sum_{n=1}^{\infty}}\varphi_{n}^{1,\varepsilon}\left(x\right)\phi_{n}^{1}\\
\\
\displaystyle{\sum_{n=1}^{\infty}}\varphi_{n}^{2,\varepsilon}\left(x\right)\phi_{n}^{2}
\end{array}\right),\quad
K^{\varepsilon}\left(x,\xi,U^{\varepsilon}\right)=\left(\begin{array}{c}
{\displaystyle \sum_{n=1}^{\infty}}\kappa_{n}^{1,\varepsilon}\left(x-\xi\right)\left\langle \mathcal{F}_{1}\left(\xi,u_{1}^{\varepsilon},u_{2}^{\varepsilon}\right),\phi_{n}^{1}\right\rangle \phi_{n}^{1}\\
\\
{\displaystyle \sum_{n=1}^{\infty}}\kappa_{n}^{2,\varepsilon}\left(x-\xi\right)\left\langle \mathcal{F}_{2}\left(\xi,u_{1}^{\varepsilon},u_{2}^{\varepsilon}\right),\phi_{n}^{2}\right\rangle \phi_{n}^{2}
\end{array}\right),
\]
and
\[
\varphi_{n}^{i,\varepsilon}\left(x\right)=\left(\Psi_{n,k}^{i,\beta}\left(x\right)+\frac{e^{-\sqrt{\alpha_{i}\lambda_{n}^{i}}x}}{2}\right)\left\langle u_{0}^{i,\varepsilon},\phi_{n}^{i}\right\rangle +\frac{1}{\sqrt{\alpha_{i}\lambda_{n}^{i}}}\left(\Psi_{n,k}^{i,\beta}\left(x\right)-\frac{e^{-\sqrt{\alpha_{i}\lambda_{n}^{i}}x}}{2}\right)\left\langle u_{1}^{i,\varepsilon},\phi_{n}^{i}\right\rangle,
\]
\[
\kappa_{n}^{i,\varepsilon}\left(x-\xi\right)=\frac{1}{\sqrt{\alpha_{i}\lambda_{n}^{i}}}\left(\Psi_{n,k}^{i,\beta}\left(x-\xi\right)-\frac{e^{-\sqrt{\alpha_{i}\lambda_{n}^{i}}\left(x-\xi\right)}}{2}\right),\quad i\in\left\{ 1,2\right\} .
\]

We now set the operator $\mathcal{G}:C([0,a];\mathcal{H})\to C([0,a];\mathcal{H})$ by
\[
\mathcal{G}(W)(x) = H^{\varepsilon}(x) + \int_{0}^{x}K^{\varepsilon}(x,\xi,W(\xi))d\xi,\quad \forall W\in C([0,a];\mathcal{H}).
\]

Taking a glimpse of the fact that for $W=\left(w_{1},w_{2}\right)^{t},\bar{W}=\left(\bar{w}_{1},\bar{w}_{2}\right)^{t}\in C([0,a];\mathcal{H})$ we do the subtraction
 \[
\mathcal{G}\left(W\right)\left(x\right)-\mathcal{G}\left(\bar{W}\right)\left(x\right)=\int_{0}^{x}\left(\begin{array}{c}
{\displaystyle \sum_{n=1}^{\infty}}\kappa_{n}^{1,\varepsilon}\left(x-\xi\right)\left\langle \mathcal{F}_{1}\left(\xi,w_{1},w_{2}\right)-\mathcal{F}_{1}\left(\xi,\bar{w}_{1},\bar{w}_{2}\right),\phi_{n}^{1}\right\rangle \phi_{n}^{1}\\
\\
{\displaystyle \sum_{n=1}^{\infty}}\kappa_{n}^{2,\varepsilon}\left(x-\xi\right)\left\langle \mathcal{F}_{2}\left(\xi,w_{1},w_{2}\right)-\mathcal{F}_{2}\left(\xi,\bar{w}_{1},\bar{w}_{2}\right),\phi_{n}^{2}\right\rangle \phi_{n}^{2}
\end{array}\right)d\xi,
\]
then via the definition of the $\mathcal{H}$-norm with the aid of Parseval's relation taking into account each $L^2$-norm and using H\"older's inequality, we arrive at
\begin{equation}
\left\Vert \mathcal{G}\left(W\right)\left(x\right)-\mathcal{G}\left(\bar{W}\right)\left(x\right)\right\Vert _{\mathcal{H}}^{2}\le
\sum_{n=1}^{\infty}\int_{0}^{x}\sum_{i\in\left\{ 1,2\right\} }\left|\kappa_{n}^{i,\varepsilon}\left(x-\xi\right)\right|^{2}d\xi
\int_{0}^{x}\left|\left\langle \mathcal{F}_{i}\left(\xi,w_{1},w_{2}\right)-\mathcal{F}_{i}\left(\xi,\bar{w}_{1},\bar{w}_{2}\right),\phi_{n}^{i}\right\rangle \right|^{2}d\xi.\label{eq:GG1}
\end{equation}

To explore the proof, we only need to prove that there exists $m_{0}\in\mathbb{N}$ such that the operator $\mathcal{G}^{m_0}:=\mathcal{G}(\mathcal{G}^{m_0-1})$ which also maps from $X=C([0,a];\mathcal{H})$ onto itself, is a contraction. In other words, we shall find out that there exists $K\in(0,1)$ such that
\[
\left\Vert \mathcal{G}^{m_{0}}\left(W\right)-\mathcal{G}^{m_{0}}\left(\bar{W}\right)\right\Vert _{X}\le K\left\Vert W-\bar{W}\right\Vert _{X}. 
\]

In fact, we prove by mathematical induction that for all $m\in \mathbb{N}$, the following estimate holds:
\begin{equation}
\left\Vert \mathcal{G}^{m}\left(W\right)\left(x\right)-\mathcal{G}^{m}\left(\bar{W}\right)\left(x\right)\right\Vert _{\mathcal{H}}^{2}
\le \frac{x^m}{m!}L_{\beta}^{2m}\left\Vert W-\bar{W}\right\Vert _{X}^{2},\quad \forall x\in [0,a],\label{eq:Gintro}
\end{equation}
where $L_{\beta}>0$ is a constant given by
\[
L_{\beta}^{2}=\frac{a\beta^{-2}}{2}\max\left\{ 1,\left(ka\right)^{2k}\right\} \sum_{i,j\in\left\{ 1,2\right\} }\frac{\left(\left|\gamma_{i}\right|\left|\delta_{ij}\right|+\left|\sigma_{ij}\right|\right)^{2}}{\alpha_{i}\lambda_{1}^{i}}.
\]

The base step, i.e. $m=1$, is trivial by the estimate \eqref{eq:GG1}. Using the elementary inequality
\[
(ka)^{\frac{kx}{a}}\le\max\left\{1,(ka)^{k}\right\},\quad\forall x\in [0,a],
\]
and with $\beta\in (0,1)$ sufficiently small to gain
\[
\left(\ln\left(\frac{a^{k}}{k\beta}\right)\right)^{\frac{-kx}{a}}\le 1,\quad \forall x\in [0,a],
\]
we obtain another type of boundedness of the kernel (obviously, it is not rigorous as Lemma \ref{lem:boundedness}):
\begin{equation}
\Psi_{n,k}^{i,\beta}\left(x\right)\le\frac{1}{2}\max\left\{1,(ka)^{k}\right\}\beta^{-1},
\quad\forall n\in\mathbb{N},i\in\left\{ 1,2\right\},\forall x\in [0,a].
\end{equation} 

In addition, owing to $\kappa_{n}^{i,\varepsilon}(x-\xi)\le \Psi_{n,k}^{i,\beta}\left(x\right)/\sqrt{\alpha_{i}\lambda_{1}^{i}}$ in combination with Lemma \ref{lem:lipschitzf}, we find that
\[
\left|\kappa_{n}^{i,\varepsilon}\left(x-\xi\right)\right|^{2}
\le \frac{\beta^{-2}}{4\alpha_{i}\lambda_{1}^{i}}\max\left\{1,(ka)^{2k}\right\},
\]
and
\[
\left|\left\langle \mathcal{F}_{i}\left(\xi,w_{1},w_{2}\right)-\mathcal{F}_{i}\left(\xi,\bar{w}_{1},\bar{w}_{2}\right),\phi_{n}^{i}\right\rangle \right|^{2}
\le \sum_{j\in\left\{1,2\right\}}(|\gamma_{i}||\delta_{ij}|+|\sigma_{ij}|)^2
\left|\left\langle w_{j}(\xi,\cdot)-\bar{w}_{j}(\xi,\cdot),\phi_{n}^{i}\right\rangle\right|^2.
\]

Therefore, we continue to estimate \eqref{eq:GG1} by
\begin{eqnarray}
\left\Vert \mathcal{G}\left(W\right)\left(x\right)-\mathcal{G}\left(\bar{W}\right)\left(x\right)\right\Vert _{\mathcal{H}}^{2} & \le & \frac{a\beta^{-2}}{2}\max\left\{ 1,\left(ka\right)^{2k}\right\} \sum_{i,j\in\left\{ 1,2\right\} }\frac{\left(\left|\gamma_{i}\right|\left|\delta_{ij}\right|+\left|\sigma_{ij}\right|\right)^{2}}{\alpha_{i}\lambda_{1}^{i}}\int_{0}^{x}\left\Vert W\left(\xi,\cdot\right)-\bar{W}\left(\xi,\cdot\right)\right\Vert _{\mathcal{H}}^{2}d\xi \nonumber\\
 & \le & L_{\beta}^{2}x\left\Vert W-\bar{W}\right\Vert _{X}^{2},\label{eq:GG2}
\end{eqnarray}
which closes the base case, and we state that \eqref{eq:Gintro} holds for $m=1$.

In the inductive step, we suppose that \eqref{eq:Gintro} holds for $m=r$, then prove that it also holds for $m=r+1$. Certainly, the path of estimating this case comes into the estimates as well as inequalities in the above base case. Thus, we complete this step by the direct estimate from \eqref{eq:GG2}:
\begin{eqnarray*}
\left\Vert \mathcal{G}^{r+1}\left(W\right)\left(x\right)-\mathcal{G}^{r+1}\left(\bar{W}\right)\left(x\right)\right\Vert _{\mathcal{H}}^{2} & \le & L_{\beta}^{2}\int_{0}^{x}\left\Vert \mathcal{G}^{r}\left(W\right)\left(\xi\right)-\mathcal{G}^{r}\left(\bar{W}\right)\left(\xi\right)\right\Vert _{\mathcal{H}}^{2}d\xi\\
 & \le & L_{\beta}^{2}\int_{0}^{x}L_{\beta}^{2r}\frac{\xi^{r}}{r!}\left\Vert W-\bar{W}\right\Vert _{X}^{2}d\xi\\
 & \le & \frac{x^{r+1}}{\left(r+1\right)!}L_{\beta}^{2\left(r+1\right)}\left\Vert W-\bar{W}\right\Vert _{X}^{2}.
\end{eqnarray*}

Hence, \eqref{eq:Gintro} is true for all $m\in\mathbb{N}$.

At the aforementioned point we expected, there exists $m_0\in\mathbb{N}$ such that the operator $\mathcal{G}^{m_0}$ is a contraction since
\[
\lim_{m\to\infty}\frac{x^m}{m!}L_{\beta}^{2m}=0.
\] 
As a result, this fact yields the existence and uniqueness of the solution $U^{\varepsilon}\in C([0,a];\mathcal{H})$ of the equation $W=\mathcal{G}^{m_0}(W)$. Besides, taking $\mathcal{G}$ into both sides of this equation we know that $\mathcal{G}^{m_0}(\mathcal{G}(U^{\varepsilon}))=\mathcal{G}(\mathcal{G}^{m_0}(U^{\varepsilon}))=\mathcal{G}(U^{\varepsilon})$. Henceforward, the uniqueness of the fixed point of the operator $\mathcal{G}^{m_0}$ greatly confirms $\mathcal{G}^{m_0}(U^{\varepsilon})=U^{\varepsilon}$. Accordingly, the equation $W=\mathcal{G}^{m_0}(W)$ admits a unique solution in $C([0,a];\mathcal{H})$, and we then end up with the proof of the theorem.

\end{proof}

\begin{thm}
(Stability)
\label{thm:3}If $U^{\varepsilon}=\left(u_{1}^{\varepsilon},u_{2}^{\varepsilon}\right)^{t}$ and $V^{\varepsilon}=\left(v_{1}^{\varepsilon},v_{2}^{\varepsilon}\right)^{t}$
are two pairs of solutions to (\ref{eq:uepmild}),
respectively, corresponding to the initial states $\left(U_{0}^{\varepsilon},U_{1}^{\varepsilon}\right)$
and $\left(V_{0}^{\varepsilon},V_{1}^{\varepsilon}\right)$,
then for all $0\le x\le a$, the stability estimate holds:
\[
\left\Vert U^{\varepsilon}\left(x,\cdot\right)-V^{\varepsilon}\left(x,\cdot\right)\right\Vert ^{2}_{\mathcal{H}}
\le C_{a}\left(ka\right)^{\frac{2kx}{a}}\beta^{-\frac{2x}{a}}\left(\ln\left(\frac{a^{k}}{k\beta}\right)\right)^{-\frac{2kx}{a}}\left(\left\Vert U_{0}^{\varepsilon}-V_{0}^{\varepsilon}\right\Vert ^{2}_{\mathcal{H}}
+\left\Vert U_{1}^{\varepsilon}-V_{1}^{\varepsilon}\right\Vert ^{2}_{\mathcal{H}}\right),
\]
where $C_{a}$ is a generic positive constant only depending on $a,\alpha_{1},\alpha_{2},\lambda_{1},\gamma_{1},\gamma_{2},\delta_{11},\delta_{12}, \delta_{21},\delta_{22},\sigma_{11},\sigma_{12},\sigma_{21},\\ \sigma_{22}$.
\end{thm}
\begin{proof}
From (\ref{eq:uepmild}), we take the inner product with $\phi_{n}^{i}$ and subtract each other, then deduce that for $i\in\left\{1,2\right\}$
\begin{eqnarray*}
\left|\left\langle u_{i}^{\varepsilon}\left(x,\cdot\right)-v_{i}^{\varepsilon}\left(x,\cdot\right),\phi_{n}^{i}\right\rangle \right| 
& \le & \left(\Psi_{n,k}^{i,\beta}\left(x\right)+\frac{e^{-\sqrt{\alpha_{i}\lambda_{n}^{i}}x}}{2}\right)\left|\left\langle u_{0}^{i,\varepsilon}-v_{0}^{i,\varepsilon},\phi_{n}^{i}\right\rangle \right|\\
 &  & +\frac{1}{\sqrt{\alpha_{i}\lambda_{n}^{i}}}\left|\Psi_{n,k}^{i,\beta}\left(x\right)-\frac{e^{-\sqrt{\alpha_{i}\lambda_{n}^{i}}x}}{2}\right|\left|\left\langle u_{1}^{i,\varepsilon}-v_{1}^{i,\varepsilon},\phi_{n}^{i}\right\rangle \right|\\
 &  & +\int_{0}^{x}\frac{1}{\sqrt{\alpha_{i}\lambda_{n}^{i}}}\left|\Psi_{n,k}^{i,\beta}\left(x-\xi\right)-\frac{e^{-\sqrt{\alpha_{i}\lambda_{n}^{i}}\left(x-\xi\right)}}{2}\right|\left|\left\langle \mathcal{F}_{i}\left(\xi,u_{1}^{\varepsilon},u_{2}^{\varepsilon}\right)-\mathcal{F}_{i}\left(\xi,v_{1}^{\varepsilon},v_{2}^{\varepsilon}\right),\phi_{n}^{i}\right\rangle \right|d\xi.
\end{eqnarray*}

Once again, the effect of Lemma \ref{lem:boundedness}-\ref{lem:lipschitzf} is available. In fact, by Parseval's relation in accordance with \eqref{eq:2.8-1} and \eqref{eq:2.17} and the very basic inequality $\left(a+b\right)^{2}\le2\left(a^{2}+b^{2}\right)$, we get
\begin{eqnarray}
d\left(x\right)
 & \le & \frac{1}{2}\left(\left(ka\right)^{\frac{2kx}{a}}\beta^{-\frac{2x}{a}}\left(\ln\left(\frac{a^{k}}{k\beta}\right)\right)^{-\frac{2kx}{a}}+1\right)\left\Vert U_{0}^{\varepsilon}-V_{0}^{\varepsilon}\right\Vert ^{2}_{\mathcal{H}}\nonumber \\
 &  & +\frac{1}{\alpha\lambda_{1}}\left(\left(ka\right)^{\frac{2kx}{a}}\beta^{-\frac{2x}{a}}\left(\ln\left(\frac{a^{k}}{k\beta}\right)\right)^{-\frac{2kx}{a}}+1\right)\left\Vert U_{1}^{\varepsilon}-V_{1}^{\varepsilon}\right\Vert ^{2}_{\mathcal{H}}\nonumber \\
 &  & +\int_{0}^{x}\frac{C}{\alpha\lambda_{1}}\left(\left(ka\right)^{\frac{2k\left(x-\xi\right)}{a}}\beta^{\frac{2\left(\xi-x\right)}{a}}\left(\ln\left(\frac{a^{k}}{k\beta}\right)\right)^{\frac{2k\left(\xi-x\right)}{a}}+1\right)d\left(\xi\right)d\xi,\label{eq:w}
\end{eqnarray}
where $\lambda_{1}$ is the smallest eigenvalues and for the sake of simplicity, we have denoted
\[
d\left(x\right)=\left\Vert U^{\varepsilon}\left(x,\cdot\right)-V^{\varepsilon}\left(x,\cdot\right)\right\Vert ^{2}_{\mathcal{H}},\;
\alpha=\min\left\{ \alpha_{1},\alpha_{2}\right\} >0,\;
C=\sum_{i,j\in\left\{ 1,2\right\} }\left(\left|\gamma_{i}\right|\left|\delta_{ij}\right|+\left|\sigma_{ij}\right|\right)^{2}.
\]

Assume that $\beta\in\left(0,1\right)$ is small enough such that for all $x\in\left[0,a\right]$
\[
\left(ka\right)^{\frac{2kx}{a}}\beta^{-\frac{2x}{a}}\left(\ln\left(\frac{a^{k}}{k\beta}\right)\right)^{-\frac{2kx}{a}}\ge1
\]
holds, it then follows from (\ref{eq:w}) that
\begin{eqnarray}
d\left(x\right) 
& \le & \left(ka\right)^{\frac{2kx}{a}}\beta^{-\frac{2x}{a}}\left(\ln\left(\frac{a^{k}}{k\beta}\right)\right)^{-\frac{2kx}{a}}\left[\left\Vert U_{0}^{\varepsilon}-V_{0}^{\varepsilon}\right\Vert ^{2}_{\mathcal{H}}
+\frac{2}{\alpha\lambda_{1}}\left(\left\Vert U_{1}^{\varepsilon}-V_{1}^{\varepsilon}\right\Vert ^{2}_{\mathcal{H}}\right)\right]\nonumber \\
 &  & +\frac{2C}{\alpha\lambda_{1}}\int_{0}^{x}\left(ka\right)^{\frac{2k\left(x-\xi\right)}{a}}\beta^{\frac{2\left(\xi-x\right)}{a}}\left(\ln\left(\frac{a^{k}}{k\beta}\right)\right)^{\frac{2k\left(\xi-x\right)}{a}}d\left(\xi\right)d\xi.\label{eq:d}
\end{eqnarray}

Multiplying both sides of (\ref{eq:d}) by $\left(ka\right)^{-\frac{2kx}{a}}\beta^{\frac{2x}{a}}\left(\ln\left(\frac{a^{k}}{k\beta}\right)\right)^{\frac{2kx}{a}}$
and putting 
\[w\left(x\right)=\left(ka\right)^{-\frac{2kx}{a}}\beta^{\frac{2x}{a}}\left(\ln\left(\frac{a^{k}}{k\beta}\right)\right)^{\frac{2kx}{a}}d\left(x\right),
\]
and thanks to Gronwall's inequality, we obtain
\[
w\left(x\right)\le\left[\left\Vert U_{0}^{\varepsilon}-V_{0}^{\varepsilon}\right\Vert ^{2}_{\mathcal{H}}
+\frac{2}{\alpha\lambda_{1}}\left(\left\Vert U_{1}^{\varepsilon}-V_{1}^{\varepsilon}\right\Vert ^{2}_{\mathcal{H}}\right)\right]\mbox{exp}\left(\frac{2Cx}{\alpha\lambda_{1}}\right),
\]
or accordingly, it can be written as
\begin{equation}
d\left(x\right)\le\left(ka\right)^{\frac{2kx}{a}}\beta^{-\frac{2x}{a}}\left(\ln\left(\frac{a^{k}}{k\beta}\right)\right)^{-\frac{2kx}{a}}\left[\left\Vert U_{0}^{\varepsilon}-V_{0}^{\varepsilon}\right\Vert ^{2}_{\mathcal{H}}
+\frac{2}{\alpha\lambda_{1}}\left(\left\Vert U_{1}^{\varepsilon}-V_{1}^{\varepsilon}\right\Vert ^{2}_{\mathcal{H}}
\right)\right]\mbox{exp}\left(\frac{2Cx}{\alpha\lambda_{1}}\right).\label{eq:2.18}
\end{equation}

This completes the proof of the theorem.
\end{proof}
\subsection{Error estimate}

In the scenarios of proving error estimates, one usually starts off with the question of finding the bridge where the link between the exact solution \eqref{eq:umild} and the regularized solution \eqref{eq:uepmild} is pointed out. The clue of applying the triangle inequality to handle such estimates addresses us to define the pseudo-regularized solution which corresponds to the exact data. The word "pseudo" herein comes from the fact that in real-world applications, there is no point to construct this solution due to the measurement. Essentially, we have the following bullets:
\begin{enumerate}
\item We define the pseudo-regularized solution $\mathcal{U}^{\varepsilon}:=\left(\bar{u}_{1}^{\varepsilon},\bar{u}_{2}^{\varepsilon}\right)^{t}$ in the same manner with \eqref{eq:uepmild}:
\begin{eqnarray}
\bar{u}_{i}^{\varepsilon}\left(x\right) & = & \sum_{n=1}^{\infty}\left[\left(\Psi_{n,k}^{i,\beta}\left(x\right)+\frac{e^{-\sqrt{\alpha_{i}\lambda_{n}^{i}}x}}{2}\right)\left\langle u_{0}^{i},\phi_{n}^{i}\right\rangle +\frac{1}{\sqrt{\alpha_{i}\lambda_{n}^{i}}}\left(\Psi_{n,k}^{i,\beta}\left(x\right)-\frac{e^{-\sqrt{\alpha_{i}\lambda_{n}^{i}}x}}{2}\right)\left\langle u_{1}^{i},\phi_{n}^{i}\right\rangle \right. \nonumber \\
 &  & \left. +\int_{0}^{x}\frac{1}{\sqrt{\alpha_{i}\lambda_{n}^{i}}}\left(\Psi_{n,k}^{i,\beta}\left(x-\xi\right)-\frac{e^{-\sqrt{\alpha_{i}\lambda_{n}^{i}}\left(x-\xi\right)}}{2}\right) \left\langle \mathcal{F}_{i}\left(\xi,\bar{u}_{1}^{\varepsilon},\bar{u}_{2}^{\varepsilon}\right),\phi_{n}^{i}\right\rangle d\xi\right]\phi_{n}^{i}, \nonumber \\
 &  & \quad i\in\left\{1,2\right\};\label{eq:uepmild-1}
\end{eqnarray}
\item Identify appropriate function spaces that are helpful for us to estimate the errors between \eqref{eq:umild} and \eqref{eq:uepmild-1} as well as \eqref{eq:uepmild-1} and \eqref{eq:uepmild}.
\end{enumerate}
We also remark that the error between \eqref{eq:uepmild-1} and \eqref{eq:uepmild} has been made in Theorem \ref{thm:3}. Then, with the aid of the assumption $\left(\mbox{A}_{1}\right)$, the following theorem is easy to prove.

\begin{thm}
\label{thm:est2}
Let $U^{\varepsilon}=(u_{1}^{\varepsilon},u_{2}^{\varepsilon})^{t}$ and $\mathcal{U}^{\varepsilon}=\left(\bar{u}_{1}^{\varepsilon},\bar{u}_{2}^{\varepsilon}\right)^{t}$ be the regularized solution in \eqref{eq:uepmild} and pseudo-regularized solution in \eqref{eq:uepmild-1} to \eqref{eq:umild}, respectively. Then for all $0\le x\le a$, the following estimate holds:
\[
\left\Vert U^{\varepsilon}\left(x,\cdot\right)-{\mathcal{U}}^{\varepsilon}\left(x,\cdot\right)\right\Vert ^{2}_{\mathcal{H}}
\le C_{a}\left(ka\right)^{\frac{2kx}{a}}\beta^{-\frac{2x}{a}}\varepsilon^{2}\left(\ln\left(\frac{a^{k}}{k\beta}\right)\right)^{-\frac{2kx}{a}},
\]
where $C_{a}$ is a generic positive constant only depending on $a,\alpha_{1},\alpha_{2},\lambda_{1},\gamma_{1},\gamma_{2},\delta_{11},\delta_{12},\delta_{21},\delta_{22},\sigma_{11},\sigma_{12},\sigma_{21},\\ \sigma_{22}$.
\end{thm}

It is now sufficient to prove the error between \eqref{eq:umild} and \eqref{eq:uepmild-1}. When doing so, one concerns the second point above. As a matter of fact, the Gevrey-type classes as introduced in Section \ref{sec:2} are considered rigorously.

\begin{thm}
\label{thm:4}
Suppose that the mild solution $U$ whose entries are defined in (\ref{eq:umild}) satisfies $U\in C\left(\left[0,a\right];\mathbb{G}_{\nu}^{s_{1}}\right),U_{x}\in C\left(\left[0,a\right];\mathbb{G}_{\nu}^{s_{2}}\right)$
with
\begin{equation}
\nu_{1}\ge a\sqrt{\alpha_{1}},
\nu_{2}\ge a\sqrt{\alpha_{2}}
\;\mbox{and}\;
s_{1}=k,s_{2}=k-1.\label{cond:gevrey}
\end{equation}
Then for $\beta\in\left(0,1\right)$ the
error estimate between $U$
and $\mathcal{U}^{\varepsilon}$ in \eqref{eq:uepmild-1}
is uniformly given by
\[
\left\Vert U\left(x,\cdot\right)-\mathcal{U}^{\varepsilon}\left(x,\cdot\right)\right\Vert ^{2}_{\mathcal{H}}
\le C_{a}\left(ka\right)^{\frac{2kx}{a}}\beta^{2\left(1-\frac{x}{a}\right)}\left(\ln\left(\frac{a^{k}}{k\beta}\right)\right)^{-\frac{2kx}{a}},
\]
where $C_{a}$ is a generic positive constant only depending on $a,k,\alpha_{1},\alpha_{2},\lambda_{1},\gamma_{1},\gamma_{2},\delta_{11},\delta_{12},\delta_{21},\delta_{22},\sigma_{11},\sigma_{12},\sigma_{21},\\ \sigma_{22}$.\end{thm}
\begin{proof}
Observe (\ref{eq:umild}), it is easy to verify
that
\begin{equation}
\left\langle u_{i}\left(x,\cdot\right),\phi_{n}^{i}\right\rangle +\frac{\left\langle u^{i}_{x}\left(x,\cdot\right),\phi_{n}^{i}\right\rangle }{\sqrt{\alpha_{i}\lambda_{n}^{i}}}=e^{\sqrt{\alpha_{i}\lambda_{n}^{i}}x}\left(\left\langle u_{0}^{i},\phi_{n}^{i}\right\rangle +\frac{\left\langle u_{1}^{i},\phi_{n}^{i}\right\rangle }{\sqrt{\alpha_{i}\lambda_{n}^{i}}}+\int_{0}^{x}\frac{e^{-\sqrt{\alpha_{i}\lambda_{n}^{i}}\xi}}{\sqrt{\alpha_{i}\lambda_{n}^{i}}}\left\langle \mathcal{F}_{i}\left(\xi,u_{1},u_{2}\right),\phi_{n}^{i}\right\rangle d\xi\right),i\in\left\{1,2\right\}.\label{eq:biendoi}
\end{equation}

Putting $d_{i}\left(x\right)=u_{i}\left(x\right)-\bar{u}_{i}^{\varepsilon}\left(x\right)$ for $i\in\left\{1,2\right\}$, we use \eqref{eq:biendoi}
 to obtain the following expression:
\begin{eqnarray}
d_{i}\left(x\right) & = & \sum_{n=1}^{\infty}\frac{\beta\sqrt{\alpha_{i}^{k}\left|\lambda_{n}^{i}\right|^{k}}}{2\beta\sqrt{\alpha_{i}^{k}\left|\lambda_{n}^{i}\right|^{k}}+2e^{-\sqrt{\alpha_{i}\lambda_{n}^{i}}x}}\left(\left\langle u_{i}\left(x,\cdot\right),\phi_{n}^{i}\right\rangle 
+\frac{\left\langle u_{x}^{i}\left(x,\cdot\right),\phi_{n}^{i}\right\rangle }{\sqrt{\alpha_{i}\lambda_{n}^{i}}}\right)\phi_{n}^{i}\nonumber\\
 &  & +\sum_{n=1}^{\infty}\int_{0}^{x}\frac{1}{\sqrt{\alpha_{i}\lambda_{n}^{i}}}\left(\Psi_{n,k}^{i,\beta}\left(x-\xi\right)-\frac{e^{-\sqrt{\alpha_{i}\lambda_{n}^{i}}\left(x-\xi\right)}}{2}\right)\left\langle \mathcal{F}_{i}\left(\xi,u_{1},u_{2}\right)-\mathcal{F}_{i}\left(\xi,\bar{u}_{1}^{\varepsilon},\bar{u}_{2}^{\varepsilon}\right),\phi_{n}^{i}\right\rangle d\xi\phi_{n}^{i}.\label{eq:biendoi1}
\end{eqnarray}

Therefore, by Parseval's relation, we
get from \eqref{eq:biendoi1} that
\begin{eqnarray}
\left\Vert d_{i}\left(x,\cdot\right)\right\Vert ^{2}_{L^2(0,b)}
& \le & 2\beta^{2}\sum_{n=1}^{\infty}\left|\Psi_{n,k}^{i,\beta}\left(x\right)\right|^{2}e^{2\sqrt{\alpha_{i}\lambda_{n}^{i}}\left(a-x\right)}\left(\sqrt{\alpha_{i}^{k}\left|\lambda_{n}^{i}\right|^{k}}\left\langle u_{i}\left(x,\cdot\right),\phi_{n}^{i}\right\rangle +\sqrt{\alpha_{i}^{k-1}\left|\lambda_{n}^{i}\right|^{k-1}}\left\langle u^{i}_{x}\left(x,\cdot\right),\phi_{n}^{i}\right\rangle \right)^{2}\nonumber\\
 &  & +2a^{2}\sum_{n=1}^{\infty}\int_{0}^{x}\frac{1}{\alpha_{i}\lambda_{n}^{i}}\left(\Psi_{n,k}^{i,\beta}\left(x-\xi\right)-\frac{e^{-\sqrt{\alpha_{i}\lambda_{n}^{i}}\left(x-\xi\right)}}{2}\right)^{2} \nonumber \\
 &  &\left|\left\langle \mathcal{F}_{i}\left(\xi,u_{1},u_{2}\right)-\mathcal{F}_{i}\left(\xi,\bar{u}_{1}^{\varepsilon},\bar{u}_{2}^{\varepsilon}\right),\phi_{n}^{i}\right\rangle \right|^{2}d\xi,\label{eq:dd1}
\end{eqnarray}
and (\ref{eq:filter}) together
with (\ref{eq:2.8-1}) and (\ref{eq:2.10}), the right-hand side of \eqref{eq:dd1} can be estimated by the following term:
\begin{eqnarray}
\beta^{2}\left(ka\right)^{\frac{2kx}{a}}\beta^{-\frac{2x}{a}}\left(\ln\left(\frac{a^{k}}{k\beta}\right)\right)^{-\frac{2kx}{a}}\alpha_{i}^{k-1}\sum_{n=1}^{\infty}e^{2\sqrt{\alpha_{i}\lambda_{n}^{i}}\left(a-x\right)}\left(\alpha_{i}\left|\lambda_{n}^{i}\right|^{k}\left|\left\langle u_{i}\left(x,\cdot\right),\phi_{n}^{i}\right\rangle \right|^{2}+\left|\lambda_{n}^{i}\right|^{k-1}\left|\left\langle u^{i}_{x}\left(x,\cdot\right),\phi_{n}^{i}\right\rangle \right|^{2}\right)\nonumber \\
+\frac{a^{2}}{\alpha_{i}\lambda_{1}}\int_{0}^{x}\left(ka\right)^{\frac{k\left(x-\xi\right)}{a}}\beta^{\frac{\xi-x}{a}}\left(\ln\left(\frac{a^{k}}{k\beta}\right)\right)^{\frac{k\left(\xi-x\right)}{a}}\left\Vert \mathcal{F}_{i}\left(\xi,u_{1},u_{2}\right)-\mathcal{F}_{i}\left(\xi,\bar{u}_{1}^{\varepsilon},\bar{u}_{2}^{\varepsilon}\right)\right\Vert ^{2}d\xi.\label{eq:dddd}
\end{eqnarray}

Thanks to the global Lipschitz property (\ref{eq:2.17}) of $\mathcal{F}_{i}$ we employ the Gevrey-type assumptions to combine \eqref{eq:dd1}-\eqref{eq:dddd}. Essentially, after taking summation in $i$ we arrive at
\begin{eqnarray}
\left\Vert U\left(x,\cdot\right)-\mathcal{U}^{\varepsilon}\left(x,\cdot\right)\right\Vert ^{2}_{\mathcal{H}}
 & \le & \beta^{2}\left(ka\right)^{\frac{2kx}{a}}\beta^{-\frac{2x}{a}}\left(\ln\left(\frac{a^{k}}{k\beta}\right)\right)^{-\frac{2kx}{a}}\bar{\alpha}\left(\alpha_{1}^{k-1}+\alpha_{2}^{k-1}\right)e^{a-x}\left(\left\Vert U\right\Vert _{C\left(\left[0,a\right];\mathbb{G}_{\nu}^{s_{1}}\right)}^{2}+\left\Vert U_{x}\right\Vert _{C\left(\left[0,a\right];\mathbb{G}_{\nu}^{s_{2}}\right)}^{2}\right)\nonumber \\
 &  & +\frac{2a^{2}}{\lambda_{1}}\left(\frac{C_{1}}{\alpha_{1}}+\frac{C_{2}}{\alpha_{2}}\right)\int_{0}^{x}\left(ka\right)^{\frac{k\left(x-\xi\right)}{a}}\beta^{\frac{\xi-x}{a}}\left(\ln\left(\frac{a^{k}}{k\beta}\right)\right)^{\frac{k\left(\xi-x\right)}{a}}\left\Vert U\left(\xi,\cdot\right)-\mathcal{U}^{\varepsilon}\left(\xi,\cdot\right)\right\Vert ^{2}_{\mathcal{H}}d\xi,
 \label{eq:aa}
\end{eqnarray}
where we have used the fact that 
$\left\Vert U\left(x,\cdot\right)-\mathcal{U}^{\varepsilon}\left(x,\cdot\right)\right\Vert ^{2}_{\mathcal{H}}=
\left\Vert d_{1}\left(x,\cdot\right)\right\Vert^{2}_{L^2(0,b)}+\left\Vert d_{2}\left(x,\cdot\right)\right\Vert^{2}_{L^2(0,b)}$ and denoted $\bar{\alpha}=\max\left\{\alpha_{1},\alpha_{2}\right\}$ and
\[
C_{1}=\left(\left|\gamma_{1}\right|\left|\delta_{11}\right|+\left|\sigma_{11}\right|\right)^{2}+\left(\left|\gamma_{1}\right|\left|\delta_{12}\right|+\left|\sigma_{12}\right|\right)^{2},
\]
\[
C_{2}=\left(\left|\gamma_{2}\right|\left|\delta_{21}\right|+\left|\sigma_{21}\right|\right)^{2}+\left(\left|\gamma_{2}\right|\left|\delta_{22}\right|+\left|\sigma_{22}\right|\right)^{2}.
\]

On the other hand, (\ref{eq:aa}) can be rewritten as
\[
w\left(x\right)\le\beta^{2}\bar{\alpha}\left(\alpha_{1}^{k-1}+\alpha_{2}^{k-1}\right)e^{a-x}C_{a}+\frac{2a^{2}}{\lambda_{1}}\left(\frac{C_{1}}{\alpha_{1}}+\frac{C_{2}}{\alpha_{2}}\right)\int_{0}^{x}w\left(\xi\right)d\xi.
\]
where we have denoted
\[
w\left(x\right)=\left(ka\right)^{-\frac{2kx}{a}}\beta^{\frac{2x}{a}}\left(\ln\left(\frac{a^{k}}{k\beta}\right)\right)^{\frac{2kx}{a}}\left\Vert U\left(x,\cdot\right)-\mathcal{U}^{\varepsilon}\left(x,\cdot\right)\right\Vert ^{2}_{\mathcal{H}},
\]
\[
C_{a}=\left\Vert U\right\Vert _{C\left(\left[0,a\right];\mathbb{G}_{\nu}^{s_{1}}\right)}^{2}+\left\Vert U_{x}\right\Vert _{C\left(\left[0,a\right];\mathbb{G}_{\nu}^{s_{2}}\right)}^{2}.
\]

By Gronwall's inequality, we thus obtain
\[
\left\Vert U\left(x,\cdot\right)-\mathcal{U}^{\varepsilon}\left(x,\cdot\right)\right\Vert ^{2}_{\mathcal{H}}
\le C_{a}\bar{\alpha}\left(\alpha_{1}^{k-1}+\alpha_{2}^{k-1}\right)\left(ka\right)^{\frac{2kx}{a}}\beta^{2\left(1-\frac{x}{a}\right)}\left(\ln\left(\frac{a^{k}}{k\beta}\right)\right)^{-\frac{2kx}{a}}\mbox{exp}\left(\frac{2a^{2}x}{\lambda_{1}}\left(\frac{C_{1}}{\alpha_{1}}+\frac{C_{2}}{\alpha_{2}}\right)+a-x\right),
\]
which leads to the proof of the theorem.\end{proof}
\begin{thm}
\label{thm:5}
(Error estimate)
If $\beta=\varepsilon^{m}$ for $m\in\left(0,1\right]$,
then the error estimate between the
exact solution $U$ given by (\ref{eq:umild})
and the regularized solution $U^{\varepsilon}$
given by (\ref{eq:uepmild}) is
\[
\left\Vert u\left(x\right)-u^{\varepsilon}\left(x\right)\right\Vert ^{2}+\left\Vert v\left(x\right)-v^{\varepsilon}\left(x\right)\right\Vert ^{2}\le C_{a}\left(ka\right)^{\frac{2kx}{a}}\left(\varepsilon^{2m\left(1-\frac{x}{a}\right)}+\varepsilon^{2\left(1-\frac{mx}{a}\right)}\right)\left(\ln\left(\frac{a^{k}}{k\beta}\right)\right)^{-\frac{2kx}{a}},
\]
for all $0\le x\le a$, where $C_{a}$ is a positive constant only
depending on $a,k,\alpha_{1},\alpha_{2},\lambda_{1},\gamma_{1},\gamma_{2},\delta_{11},\delta_{12},\delta_{21},\delta_{22},\\ \sigma_{11},\sigma_{12},\sigma_{21},\sigma_{22}$.

As a consequence, we have $u_{i}^{\varepsilon}(x,\cdot)\to u_{i}(x,\cdot)$ strongly in $L^2(0,b)$ for all $0\le x\le a$ and $i\in\left\{1,2\right\}$ as $\varepsilon\to 0^{+}$.
\end{thm}
\begin{proof}
The proof is straightforward by Theorem \ref{thm:est2}-\ref{thm:4}
. Indeed, choosing $\beta=\varepsilon^{m},m\in\left(0,1\right]$
we apply the triangle inequality and the elementary inequality $(a+b)^2\le 2\left(a^2 + b^2\right)$ to obtain
\begin{eqnarray*}
\left\Vert U\left(x,\cdot\right)-U^{\varepsilon}\left(x,\cdot\right)\right\Vert ^{2}_{\mathcal{H}}
 & \le & 2\left(\left\Vert U\left(x,\cdot\right)-\mathcal{U}^{\varepsilon}\left(x,\cdot\right)\right\Vert ^{2}_{\mathcal{H}}
 +\left\Vert \mathcal{U}^{\varepsilon}\left(x,\cdot\right)-U^{\varepsilon}\left(x,\cdot\right)\right\Vert ^{2}_{\mathcal{H}}\right)
\\
 & \le & C_{a}\left(ka\right)^{\frac{2kx}{a}}\left(\varepsilon^{2m\left(1-\frac{x}{a}\right)}+\varepsilon^{2\left(1-\frac{mx}{a}\right)}\right)\left(\ln\left(\frac{a^{k}}{k\varepsilon^{m}}\right)\right)^{-\frac{2kx}{a}}.
\end{eqnarray*}

This ends the proof of the theorem.
\end{proof}

\begin{rem}
Even though the well-posedness and convergence rate have been made, some remarks should be singled out:
\begin{enumerate}
\item Theorem \ref{thm:est2} shows a prototype to estimate the error between $U$ and $U^{\varepsilon}$ under Gevrey-type assumptions on the exact solution. The mild restriction \eqref{cond:gevrey} is flexible and can be changed properly in some certain cases. For instance, if $\lambda_{1}\ge 1$ it allows us to extend this condition by setting $s_{1}\ge k$ and $s_{2}\ge k-1$. It also serves several analytic functions if we choose $k=1$;

\item The proposed method is shown
to be good convergence in comparison with well-known methods, such
as quasi-reversibility method and quasi-boundary value method. In
fact, one of the superficial advantages is described by the error estimate
at $x=a$. Particularly, it is valid in the logarithmic type, i.e.
\[
\left\Vert U\left(a,\cdot\right)-U^{\varepsilon}\left(a,\cdot\right)\right\Vert ^{2}_{\mathcal{H}}
\le C_{a}\left(ka\right)^{2k}\left(1+\varepsilon^{2\left(1-m\right)}\right)\left(\ln\left(\frac{a^{k}}{k\varepsilon^{m}}\right)\right)^{-2k}\le C_{a}\left(\ln\left(\frac{a^{k}}{k\varepsilon^{m}}\right)\right)^{-2k},\quad m\in\left(0,1\right],
\]
which cannot happen in the study of quasi-reversibility method (\cite{LL67}).

Moreover, the convergence rate is quite general. In principle, it
is of the order $\mathcal{O}\left(\varepsilon^{m}\left(\ln\left(\frac{a^{k}}{k\varepsilon^{m}}\right)\right)^{-\frac{kx}{a}}\right)$
and this is a generalization of many previous results, \emph{e.g.}
\cite{TTTK15,TTK15}.
\end{enumerate}
\end{rem}

\section{A computational tool}
\label{sec:compute}
\subsection{Preliminaries}
This part of work concentrates on a tool of implementation for the stable approximate solution \eqref{eq:uepmild}. Due to the orthonormal basis $\left\{\phi_n\right\}_{n\in\mathbb{N}}$, we focus here on computing the real-valued Fourier coefficient, denoted by $U_{n}^{\varepsilon}$, at each frequency $n\in\mathbb{N}$, which reads
\begin{eqnarray}
\left\langle u_{i}^{\varepsilon}\left(x,\cdot\right),\phi_{n}^{i} \right\rangle
& = & \left(\Psi_{n,k}^{i,\beta}\left(x\right)+\frac{e^{-\sqrt{\alpha_{i}\lambda_{n}^{i}}x}}{2}\right)\left\langle u_{0}^{i,\varepsilon},\phi_{n}^{i}\right\rangle +\frac{1}{\sqrt{\alpha_{i}\lambda_{n}^{i}}}\left(\Psi_{n,k}^{i,\beta}\left(x\right)-\frac{e^{-\sqrt{\alpha_{i}\lambda_{n}^{i}}x}}{2}\right)\left\langle u_{1}^{i,\varepsilon},\phi_{n}^{i}\right\rangle 
\nonumber \\
&  & +\int_{0}^{x}\frac{1}{\sqrt{\alpha_{i}\lambda_{n}^{i}}}\left(\Psi_{n,k}^{i,\beta}\left(x-\xi\right)-\frac{e^{-\sqrt{\alpha_{i}\lambda_{n}^{i}}\left(x-\xi\right)}}{2}\right)\left\langle \mathcal{F}_{i}\left(\xi,u_{1}^{\varepsilon},u_{2}^{\varepsilon}\right),\phi_{n}^{i}\right\rangle d\xi,\quad i \in\left\{1,2\right\}.\label{eq:uepmild-inner}
\end{eqnarray}

Even though the stability in $\mathcal{H}$ of the solution has been proved, a lot of issues concerning numerical approximations remain questioned. The fact is that the inner product is one of the obstacles since it is viewed as the highly oscillatory integral. Nowadays, many efforts try to tackle the very general integral with the kernel expressed in the form of an
imaginary exponential function, i.e.
\begin{equation}
I[f]=\int_{\Omega}f(x)e^{i\omega g(x)}dx,\label{eq:newint}
\end{equation}
where $f$ and $g$ are non-oscillatory functions and $\omega$ the large number represents the frequency of oscillations. The consideration of this integral is obvious because by taking the real and imaginary parts, it reveals
two integrals with trigonometric kernels which occur very often in practice. In the present paper, we, however, study the simple case. In particular, for an open parallelepiped $\Omega = (0,a_{1})\times...\times(0,a_{d})\subset\mathbb{R}^{d}$, we recall the Sturm-Liouville problem
\[
-\Delta \phi = \lambda\phi,
\]
associated with the zero Neumann boundary conditions, and it gives us that
\[
\phi_{n} = \prod_{j=1}^{d}\sqrt{\frac{2}{a_j}}\cos\left(\frac{\pi n_j}{a_j}x_{j}\right),
\quad
\lambda_{n}=\sum_{j=1}^{d}\left(\frac{\pi n_j}{a_j}\right)^2,\quad n_{j}\in\mathbb{N},j\in\left\{1,...,d\right\}.
\]

Despite the general form, we herein consider $d=1$, and \eqref{eq:newint} thus reduces to the one-dimensional case with the cosine kernel:
\begin{equation}
I[f]=\int_{0}^{b}f(y)\cos\left(\frac{\pi n}{b}y\right)dy,\label{eq:intconsider}
\end{equation}
which implies the standard Fourier oscillator with $g(y)=y$ and $\omega = \pi n/b$.

Concerning the motivation of studying such types of integrals, the reader can be referred to \cite{Olv08}. We also remind that the non-oscillator function $f$ in the case being tackled \eqref{eq:intconsider} is drastically affected by the noise level $\varepsilon$. On the other hand, by measurement only discrete values of such input are known, and the measured data $\mbox{W}^{\varepsilon}$ can be expressed from the exact data $\mbox{W}$, as follows:
\begin{equation}
\mbox{W}^{\varepsilon}\left(y_{j}\right)=\mbox{W}\left(y_{j}\right)+\varepsilon\mbox{rand}\left(y_{j}\right),\label{eq:noisedata}
\end{equation}
where $\mbox{rand}(\cdot)$ performs the random generator and shall be chosen properly to suit the assumption $\left(\mbox{A}_{1}\right)$.
Until now, a vast number of approaches attempted to deal with the highly oscillatory integrals. In this work, we are interested in the use of the asymptotic expansion and the Filon-type methods.

\subsection{The asymptotic expansions and the Filon-type methods}\label{subsec:filon}
The asymptotic expansion is a simply direct approach which is used for the stationary-free functions $g$ (i.e. $g'\ne 0$). By integration by parts, \eqref{eq:intconsider} in fact becomes
\begin{equation}
I[f]=\frac{1}{\omega}f(b)\sin(\omega b)
- \frac{1}{\omega}\int_{0}^{b}f'(y)\sin(\omega y)dy. \label{eq:I1}
\end{equation}

We, however, stress that the asymptotic expansions do not, in general, converge when the frequency is fixed. With the absolute accuracy $\mathcal{O}(\omega^{-2})$ (cf. \cite{Olv08}), this approach shall be highly useful when $n$ is huge since it reduces the computational costs a lot. We may point out the limitation by choosing $b=\pi$. The first term on the right-hand side of \eqref{eq:I1} vanishes immediately no matter what $f$ is. On the other side, continuing to choose $f(y)=y$ and $n=3$ the exact value of $I[f]$ yields $-2/9\approx -0.2222$ whilst the approximation by the expansion is definitely $0$. This comparison also agrees with the aforementioned error.

While the asymptotic expansions merely do a good approximation for an extremely large level of oscillations, the Filon-type method in which we are interested next, extending the work of Filon \cite{Filon} and investigated by Iserles et al. \cite{IN05}, works well. For the sake of clarity, we shall re-introduce shortly all computational procedures of this method in our context.

Let $\left\{y_{l}\right\}_{0}^{\nu}$ be a set of arbitrary node points such that
\begin{equation}
0=y_{0}<y_{1}<...<y_{\nu}=b,\label{eq:point1}
\end{equation}
and let $\left\{m_{l}\right\}_{0}^{\nu}$ be a set of multiplicities associated with those prescribed points.
We shall construct a polynomial $p$:
\begin{equation}
p(y) = \sum_{q=0}^{\ell}c_{q}y^{q}
= c_{0}+c_{1}y+...+c_{\ell}y^{\ell},\label{eq:qori}
\end{equation}
whose degree is defined by
\[
\ell:=\sum_{l=0}^{\nu}m_{l}-1.
\]

Finding the function $p$ in \eqref{eq:qori} is not problematic. Indeed, the need is to determine the constant $c_{q}$ and by using the nodes $y_{l}$, we arrive at the following system:
\begin{equation}
\begin{cases}
p\left(y_{l}\right)=f\left(y_{l}\right),\\
p'\left(y_{l}\right)=f'\left(y_{l}\right),\\
\vdots\\
p^{\left(m_{l}-1\right)}\left(y_{l}\right)=f^{\left(m_{l}-1\right)}\left(y_{l}\right),
\end{cases}
\label{eq:systemfl1}
\end{equation}
for every $0\le l\le \nu$.
Then, we obtain a good approximation of \eqref{eq:intconsider} with asymptotic order $s$ in the sense that
\begin{equation}
I[f] - \sum_{q=0}^{\ell}c_{q}I[y^{q}] \sim \mathcal{O}(\omega^{-s-1}),\label{eq:filonloai1}
\end{equation}
where $s:=\min\left\{m_{0},m_{\nu}\right\}$.

This method, however, is not really handy when the structure of $f$ is complicated. One can apply it to compute the two inner products $\left\langle u_{0}^{i,\varepsilon},\phi_{n}^{i}\right\rangle$ and $\left\langle u_{1}^{i,\varepsilon},\phi_{n}^{i}\right\rangle$ in \eqref{eq:uepmild-inner} whilst the term $\left\langle \mathcal{F}_{i}\left(\xi,u_{1}^{\varepsilon},u_{2}^{\varepsilon}\right),\phi_{n}^{i}\right\rangle$ is still troublesome. It is simply because that the problem \eqref{eq:uepmild-inner} must be solved numerically by an iteration-like method. For each loop of iteration, we must follow the former loops by plugging it into the new loop, which basically escalates the complexity of computations. We also remark from the Filon-type method above that the more the accuracy is obtained, the more derivatives we have to compute. Therefore, we are led to the derivative-free method or the adaptive Filon-type method, postulated by Iserles and N\o{}rsett in \cite{IN04}, which allows us to gain high asymptotic orders without any attention to derivatives.

The notion of this method is that we intentionally choose the interpolation nodes which depends on the frequency $\omega$ in place of the arbitrary points as \eqref{eq:point1}. In particular, the derivatives $f^{(r)}\left(y_{l}\right)$ for $r\in \left\{0,1,...,m_{l}-1\right\}$ are substituted by just the values $f\left(y_{l}+\gamma_{l,j}/\omega\right)$ for $j\in \left\{0,1,...,m_{l}-1\right\}$. Henceforward, for each $0\le l\le \nu$ the system \eqref{eq:systemfl1} becomes:
\begin{itemize}
\item For $l=0$:
\[
\begin{cases}
p\left(y_{0}\right)=f\left(y_0\right),\\
p'\left(y_{0}\right)=\left|h_{0}^{\omega}\right|^{-1}\Delta_{h_{0}^{\omega}}^{1}[f]\left(y_0\right),\\
\vdots\\
p^{\left(m_{0}-1\right)}\left(y_{0}\right)=\left|h_{0}^{\omega}\right|^{-m_0+1}\Delta_{h_{0}^{\omega}}^{m_0-1}[f]\left(y_0\right),
\end{cases}
\]
\item For $l\in\left\{1,...,\nu-1\right\}$:
\[
\begin{cases}
p\left(y_{l}\right)=f\left(y_{l}\right),\\
p'\left(y_{l}\right)=\left|h_{l}^{\omega}\right|^{-1}\delta_{h_{l}^{\omega}}^{1}[f](y_{l}),\\
\vdots\\
p^{\left(m_{l}-1\right)}\left(y_{l}\right)=\left|h_{l}^{\omega}\right|^{-m_l+1}\delta_{h_{l}^{\omega}}^{m_l-1}[f](y_{l}),
\end{cases}
\]
\item For $l=\nu$:
\[
\begin{cases}
p\left(y_{\nu}\right)=f\left(y_{\nu}\right),\\
p'\left(y_{\nu}\right)=\left|h_{\nu}^{\omega}\right|^{-1}\nabla_{h_{\nu}^{\omega}}^{1}[f]\left(y_{\nu}\right),\\
\vdots\\
p^{\left(m_{l}-1\right)}\left(y_{\nu}\right)=\left|h_{\nu}^{\omega}\right|^{-m_{\nu}+1}\nabla_{h_{\nu}^{\omega}}^{m_{\nu}-1}[f]\left(y_{\nu}\right),
\end{cases}
\]
\end{itemize}
where $\Delta^{j}_{h_{0}^{\omega}}$, $\delta^{j}_{h_{l}^{\omega}}$ and $\nabla^{j}_{h_{\nu}^{\omega}}$, $j\in\left\{1,...,m_l-1\right\}$ denote the high-order forward, central and backward finite difference, respectively, and are defined by
\[
\Delta^{j}_{h_{0}^{\omega}}[f](z):=
\sum_{i=0}^{j}\left(-1\right)^{i}\left(\begin{array}{c}
j\\
i
\end{array}\right)f\left(z+\left(j-i\right)h_{0}^{\omega}\right),
\]
\[
\delta^{j}_{h_{l}^{\omega}}[f](z):=
\sum_{i=0}^{j}(-1)^{i}\left(\begin{array}{c}
j\\
i
\end{array}\right)f\left(z+\left(\frac{j}{2}-i\right)h_{l}^{\omega}\right),
\]
\[
\nabla_{h_{\nu}^{\omega}}^{j}[f]\left(z\right):=\sum_{i=0}^{j}\left(-1\right)^{i}\left(\begin{array}{c}
j\\
i
\end{array}\right)f\left(z-ih_{\nu}^{\omega}\right),\quad h_{l}^{\omega}:= \frac{\gamma_{l}}{\omega},
\]
with $\gamma_{l}>0$, $l\in\left\{0,1,...,\nu\right\}$ being free-to-choose sufficiently small constants and in addition, the ``smallness" of $\gamma_{l}$ must agree with
\[
y_0 + (m_0-1)h_{0}^{\omega}<y_{1}-(m_1-1)\frac{h_{1}^{\omega}}{2}<y_1+(m_1-1)\frac{h_{1}^{\omega}}{2}<
y_{2}-(m_2-1)\frac{h_{2}^{\omega}}{2}<...<y_{\nu-1}+(m_{\nu-1}-1)\frac{h_{\nu-1}^{\omega}}{2}< y_{\nu} - (m_{\nu}-1)h_{\nu}^{\omega},
\]
which ensures the order of interpolation. (see Figure \ref{fig:1} for the corresponding illustration)
\begin{figure}[H]
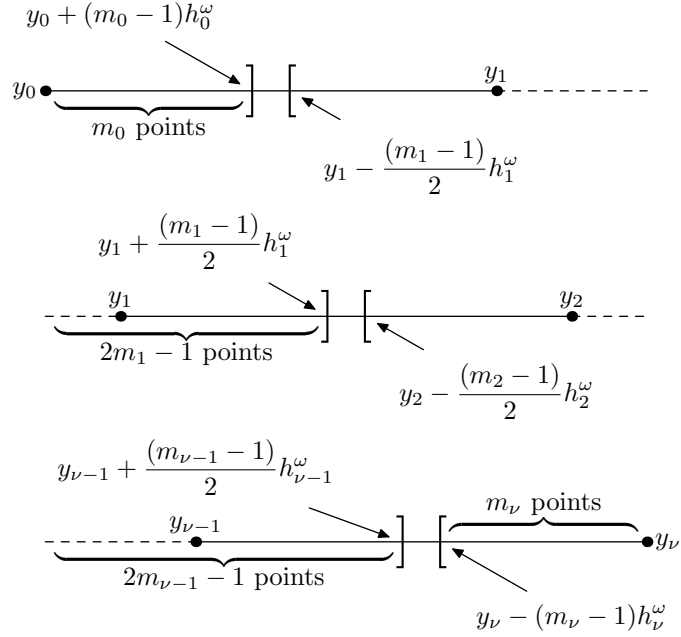

\centering
\parbox{10cm}{\convertMPtoPDF{fig.1.eps}{1}{1}}
\caption{Illustration of interpolation points by the adaptive Filon-type method.}
\label{fig:1}
\end{figure}

\begin{rem}
In \eqref{eq:filonloai1}, the moments
\[
I\left[y^{q}\right]=\int_{0}^{b}y^q\cos(\omega y)dy
\]
can be computed explicitly for each $q$ and $\omega$.
The error controls of the above Filon-type methods are the same. Additionally, it is worth noting that the construction of the Filon-type methods herein is based on the Hermite interpolation. This means that the error control strongly depends on the maximum norm of the derivative $f^{(\ell+1)}$. Once this norm becomes larger after increasing the nodes, the asymptotic expansions are better to use. Some simple examples are $f(y)=\cos(100y)$ and $f(y)=e^{10y}$. 
\end{rem}

From now on, it is possible to deduce our computational strategy in the next section:
\begin{itemize}
\item If the frequency of oscillations $\omega$ is exceedingly large, we can examine the approximation with the asymptotic expansions since it reduces the computational cost;
\item If the frequency is not very large, the adaptive Filon-method is available for all the inner products $\left\langle u_{0}^{i,\varepsilon},\phi_{n}^{i}\right\rangle$, $\left\langle u_{1}^{i,\varepsilon},\phi_{n}^{i}\right\rangle$ and $\left\langle \mathcal{F}_{i}\left(\xi,u_{1}^{\varepsilon},u_{2}^{\varepsilon}\right),\phi_{n}^{i}\right\rangle$ in \eqref{eq:uepmild-inner}.
\end{itemize}

\begin{exm}
Our selection for testing with the Filon-type method is the function
\[
f\left(y\right)=2\mbox{sech}\left(y\right)+\frac{y^{2}}{b}\mbox{sech}\left(b\right)\tanh\left(b\right).
\]

Choosing $b=1$ and the following nodes and multiplicities:
\[
y_{l}\in\left\{0,\frac{1}{2},1\right\},\quad 
m_{l}\in\left\{2,1,2\right\},
\]
which also lead to $\nu = 2$ and $\ell = 4$, the system \eqref{eq:systemfl1} becomes
\[
\begin{cases}
c_{0}=u_{1}^{1}\left(0\right),\\
c_{1}=\frac{du_{1}^{1}}{dy}\left(0\right),\\
\frac{c_{2}}{4}+\frac{c_{3}}{8}+\frac{c_{4}}{16}=u_{1}^{1}\left(\frac{1}{2}\right)-u_{1}^{1}\left(0\right)-\frac{1}{2}\frac{du_{1}^{1}}{dy}\left(0\right),\\
c_{2}+c_{3}+c_{4}=u_{1}^{1}\left(1\right)-u_{1}^{1}\left(0\right)-\frac{du_{1}^{1}}{dy}\left(0\right),\\
2c_{2}+3c_{3}+4c_{4}=\frac{du_{1}^{1}}{dy}\left(1\right)-\frac{du_{1}^{1}}{dy}\left(0\right).
\end{cases}
\]

Thus, we compute that
\[
c_0 = 2, c_{1} = 0, c_{2} = -0.5965, c_{3} = 0.3518, c_{4} = 0.0344,
\]
and with $n=1$, we find that
\[
I\left[y^{0}\right]=0, I\left[y^{1}\right] = I\left[y^{2}\right] = \frac{-2}{\pi^2},
I\left[y^{3}\right]= \frac{-3(\pi^2-4)}{\pi^4}, I\left[y^{4}\right] = \frac{-4(\pi^2 - 6)}{\pi^4}.   
\]

We then obtain the good approximation is 0.0518 whilst using the Gauss-Legendre quadrature method with 10 nodes, the approximation is 0.0524. Comparing these approximations to the round-off exact value which yields 0.0524, the Gauss-Legendre wins. However, by doing the same token with $n=10$, we compute that the approximation by the Filon-type method is $-8.4755\times 10^{-7}$ while it is terribly 0.0475, compared to the exact value $-1.5369\times 10^{-6}$, with the 10-nodes Gauss-Legendre method.

Like many quadrature methods, the more the number of nodes, the more accuracy we obtain. To verify it, we choose
\[
y_{l}\in\left\{0,\frac{1}{8},\frac{1}{4},\frac{3}{8},\frac{1}{2},\frac{5}{8},\frac{3}{4},\frac{7}{8},1\right\},
\quad
m_{l}\in\left\{2,1,1,1,1,1,1,1,2\right\},
\]
then with $n=1$ the approximation tells 0.0524. With $n=10$, it is $-1.5371\times 10^{-6}$ whereas the Gauss-Legendre requires at least 15 nodes to gain the same value.

In addition, we continue to test with the measured data, and recall that such data is defined in \eqref{eq:noisedata} where $|\mbox{rand}(y)|\le (2b)^{-1/2}$ is assumed for all $y\in [0,b]$. Notice herein that we cannot compute directly the exact inner product along with $u_{1}^{1,\varepsilon}$ due to the rand function. Nevertheless, the exact values can be used again for comparison. It is simply because that when the noise level $\varepsilon$ goes very small, the approximation must approach the exact value. We also remark that when differentiating \eqref{eq:noisedata}, the rand function does not vanish itself, and the noise (along with other random generators) as a result still appears in all the derivatives in the system \eqref{eq:systemfl1}. On the whole, with $\varepsilon = 10^{-6}$, the obtained values are 0.0524 and $-1.6455\times 10^{-6}$ for $n=1$ and $n=10$, respectively. For $n=10$, if carry on with $\varepsilon = 10^{-8}$, we get the better approximation $-1.5378\times 10^{-6}$ which agrees with our expectation.
\begin{table}
\begin{centering}
\begin{tabular}{|c|c|c|c|}
\hline 
$n$ & 1 & 5 & 10\tabularnewline
\hline 
$\varepsilon=10^{-4}$ & 2.05$\times10^{-2}$\% & 1.31\% & 4.15\%\tabularnewline
\hline 
$\varepsilon=10^{-6}$ & 2.23$\times10^{-4}$\% & 2.44$\times10^{-2}$\% & 5.52$\times10^{-2}$\%\tabularnewline
\hline 
$\varepsilon=10^{-8}$ & 6.30$\times10^{-5}$\% & 3.49$\times10^{-4}$\% & 2.18$\times10^{-4}$\%\tabularnewline
\hline 
$\varepsilon=0$ & 6.17$\times10^{-5}$\% & 1.11$\times10^{-4}$\% & 1.27$\times10^{-4}$\%\tabularnewline
\hline 
\end{tabular}
\par\end{centering}
\caption{The relative accuracy of the Filon-type method with and without the presence of the noise level.}
\label{tab:1}
\end{table}

\begin{figure}
	\centering
	\begin{subfigure}[!h]{.5\linewidth}
		\centering
		\includegraphics[scale=0.6]{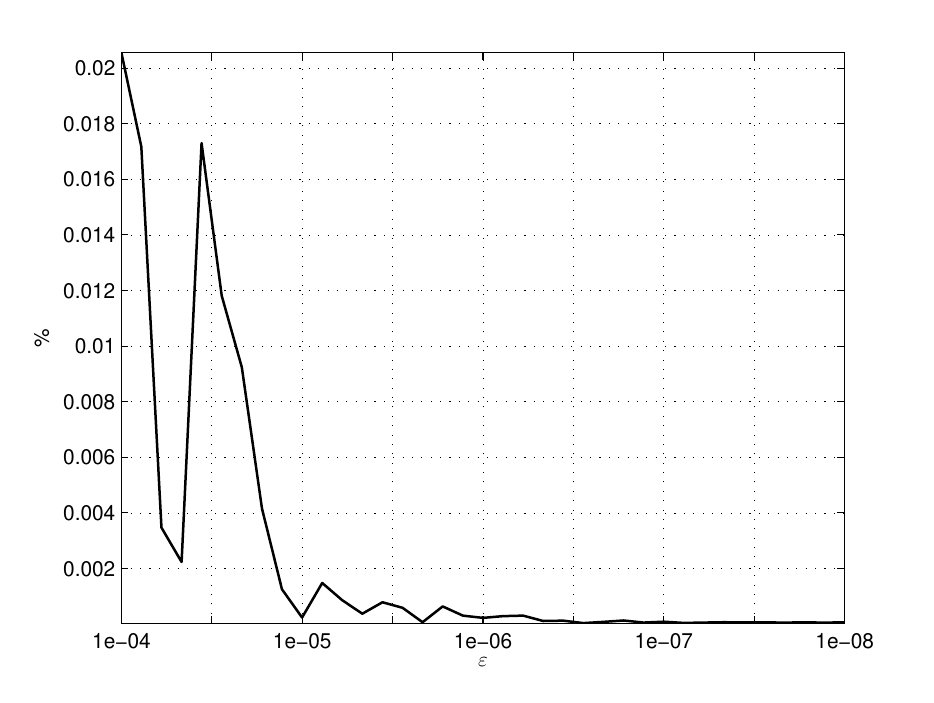}
		\caption{$n=1$}
	\end{subfigure}%
	\begin{subfigure}[!h]{.5\linewidth}
		\centering
		\includegraphics[scale=0.6]{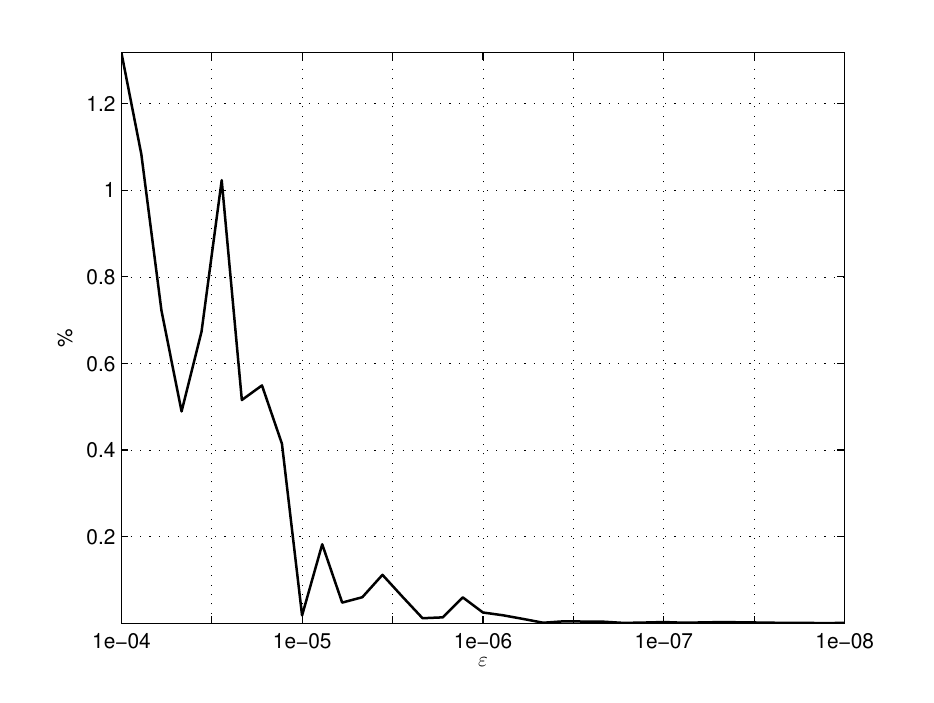}
		\caption{$n=5$}
	\end{subfigure}
	\begin{subfigure}[!h]{.5\linewidth}
		\centering
		\includegraphics[scale=0.6]{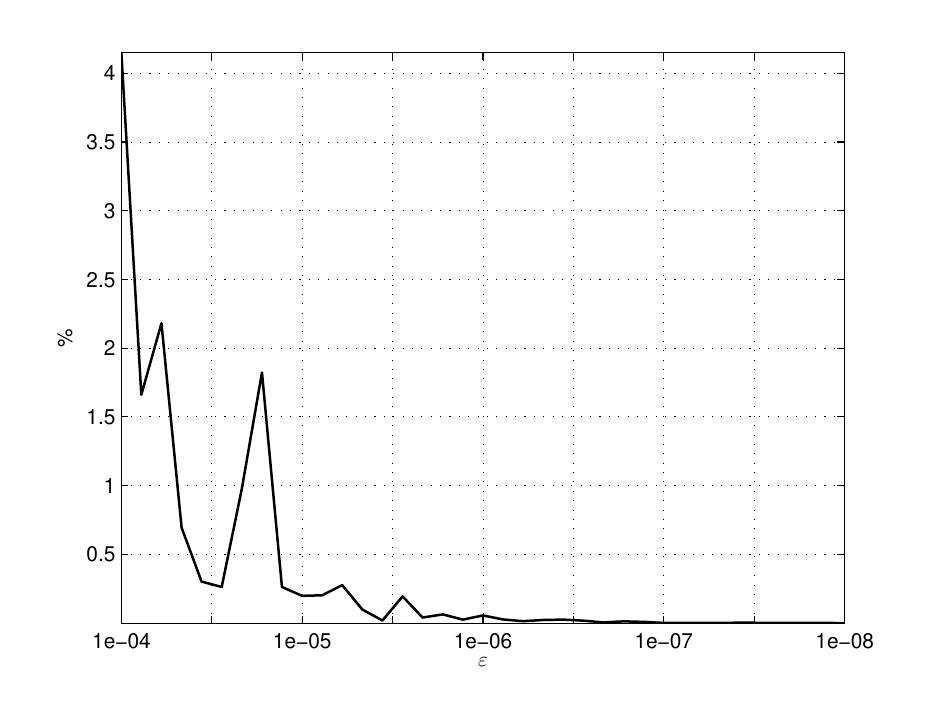}
		\caption{$n=10$}
	\end{subfigure}
	\caption{Relative errors for the approximation by the Filon-type method at $n\in\left\{1,5,10\right\}$ with various noise levels $\varepsilon$.}\label{fig:2}
\end{figure}

Owing to the decay of the oscillatory integrals (cf. \cite{Stein93}), we use the relative error in lieu of the absolute error to perceive the convergence. These errors are tabulated in Table \ref{tab:1} and illustrated in Figure \ref{fig:2}, and these at once reveal the impact of the noise on the number of frequencies. In simpler terms, at each noise level the bigger the frequency, the less the accuracy. As a consequence, the need of truncation is evident to control the practical error bound. This also facilitates our Gevrey-type priori condition, and that the convergence is thus valid.

As we may readily expect, when setting the finite-difference operator on the derivatives to deduce the adaptive Filon-type method, we must pay a price. Obviously, the price comes from those approximations. Particularly, they are of the orders $\mathcal{O}(h_0^{\omega}),\mathcal{O}(h_{\nu}^{\omega})$, respectively, for the forward and backward differences while the last operator gives the order $\mathcal{O}(\left|h_l^{\omega}\right|^2)$. Due to the presence of the noise, the ``smallness" of the quantities $\gamma_{l}$ can be computed. In fact, taking $\gamma_{l}=\varepsilon^{d}$, $d>0$ is possible. As tabulated in Table \ref{tab:1-1}, the performance of the adaptive Filon-type method under the choice $\gamma_{l}=\varepsilon$ ($d=1$) is clearly better than the results along with $\gamma_l = 10^{-4}$ for all $l\in\left\{0,1,...,\nu\right\}$ as $\varepsilon\to 0$. On top of that they are somehow better than the numerical results in Table $\ref{tab:1}$ at some noise level.

\begin{table}
\begin{centering}
\begin{tabular}{|c|c|c|c|c|c|c|}
\hline 
\multirow{2}{*}{$n$} & \multicolumn{2}{c|}{1} & \multicolumn{2}{c|}{5} & \multicolumn{2}{c|}{10}\tabularnewline
\cline{2-7} 
 & $\gamma_{l}=\varepsilon$ & $\gamma_{l}=10^{-4}$ & $\gamma_{l}=\varepsilon$ & $\gamma_{l}=10^{-4}$ & $\gamma_{l}=\varepsilon$ & $\gamma_{l}=10^{-4}$\tabularnewline
\hline 
$\varepsilon=10^{-4}$ & 2.80$\times10^{-2}$\% & 2.24$\times10^{-2}$\% & 4.19\% & 2.40\% & 5.65\% & 4.05\%\tabularnewline
\hline 
$\varepsilon=10^{-6}$ & 2.23$\times10^{-4}$\% & 2.48$\times10^{-4}$\% & 6.93$\times10^{-3}$\% & 6.32$\times10^{-3}$\% & 6.84$\times10^{-3}$\% & 2.84$\times10^{-2}$\%\tabularnewline
\hline 
$\varepsilon=10^{-8}$ & 6.22$\times10^{-5}$\% & 1.24$\times10^{-4}$\% & 3.79$\times10^{-4}$\% & 1.50$\times10^{-3}$\% & 1.32$\times10^{-4}$\% & 3.46$\times10^{-4}$\%\tabularnewline
\hline 
\end{tabular}
\par\end{centering}

\caption{Relative error comparison between the choices of $\gamma_{l}$ for
the adaptive Filon-type method with various amounts of noise levels. }
\label{tab:1-1}
\end{table}
\end{exm}

\subsection{The Picard-based iteration}\label{subsec:Picard}
We herein use the Picard-like procedure to approximate the Volterra-type integral equation \eqref{eq:uepmild}. We begin with a  mesh-grid $h_{j}^{x}=x_{j+1}-x_{j}$ for $j=\overline{0,J-1}$ and denote $h_{x}=\max_{0\le j\le J-1}\left|h_{j}^{x}\right|$. Then for each $x_{j}$ we compute $U^{\varepsilon,j+1}(y)\approx U^{\varepsilon}\left(x_{j+1},y\right)$ by the following iteration:
\begin{eqnarray}
\left\langle u_{i}^{\varepsilon,j+1},\phi_{n}^{i}\right\rangle
& = & \left(\Psi_{n,k}^{i,\beta}\left(x_{j+1}\right)+\frac{e^{-\sqrt{\alpha_{i}\lambda_{n}^{i}}x_{j+1}}}{2}\right)\left\langle u_{0}^{i,\varepsilon},\phi_{n}^{i}\right\rangle +\frac{1}{\sqrt{\alpha_{i}\lambda_{n}^{i}}}\left(\Psi_{n,k}^{i,\beta}\left(x_{j+1}\right)-\frac{e^{-\sqrt{\alpha_{i}\lambda_{n}^{i}}x_{j+1}}}{2}\right)\left\langle u_{1}^{i,\varepsilon},\phi_{n}^{i}\right\rangle
\nonumber \\
 &  & +\sum_{0\le r\le j}\int_{x_{r}}^{x_{r+1}}\frac{1}{\sqrt{\alpha_{i}\lambda_{n}^{i}}}\left(\Psi_{n,k}^{i,\beta}\left(x_{j+1}-\xi\right)-\frac{e^{-\sqrt{\alpha_{i}\lambda_{n}^{i}}\left(x_{j+1}-\xi\right)}}{2}\right)\left\langle \mathcal{F}_{i}\left(\xi,u_{1}^{\varepsilon,r},u_{2}^{\varepsilon,r}\right),\phi_{n}^{i}\right\rangle d\xi,
 \label{eq:uepmild-inner11}
\end{eqnarray}
with the starting point $u_{i}^{\varepsilon,0}=u_{i}^{\varepsilon}(x_0)\equiv u_{0}^{i,\varepsilon}$.

Note that the notation $\mathcal{F}_{i}\left(\xi,u_{1}^{\varepsilon,r},u_{2}^{\varepsilon,r}\right)$ denotes
\[
\mathcal{F}_{i}\left(\xi,u_{1}^{\varepsilon,r},u_{2}^{\varepsilon,r}\right) = 
f_{i}(\xi)-
\gamma_{i}\sin\left(\delta_{i1}u_{1}^{\varepsilon,r}+\delta_{i2}u_{2}^{\varepsilon,r}\right)
-\sigma_{i1}u_{1}^{\varepsilon,r}-\sigma_{i2}u_{2}^{\varepsilon,r},\quad i\in\left\{1,2\right\},
\]
slightly different from the usual notation we used above (for example, in \eqref{eq:uepmild}).

Normally, this iteration is linearly convergent when $h_{j}^{x}$ is sufficiently small, and here it is not an exception. Due to the boundedness of the kernel from Remark \ref{rem:bou} which increases by the reduction of $\varepsilon$,  a restriction for $h_{j}^{x}$ is required. More precisely, the choice of the discretization $h_{j}^{x}$ shall strictly depend on $\beta$ and obviously adhere to the boundedness of the kernel in Lemma \ref{lem:boundedness} and the Lipschitz constants acting on $\mathcal{F}_{i}$. In particular, the mild restriction for each loop $j$ can be designed by
\begin{equation}
\sum_{i\in\left\{ 1,2\right\} }\frac{1}{\sqrt{\alpha_{i}\lambda_{1}^{i}}}\sum_{i\in\left\{ 1,2\right\} }\left(\left|\gamma_{i}\right|+\left|\sigma_{i1}\right|+\left|\sigma_{i2}\right|\right)\left(ka\right)^{\frac{kx_{j+1}}{a}}\beta^{-\frac{x_{j+1}}{a}}\left(\ln\left(\frac{a^{k}}{k\beta}\right)\right)^{-\frac{kx_{j+1}}{a}}\le \left(h_{j}^{x}\right)^{-r},\quad r\in (0,1).
\label{eq:newres}
\end{equation}

\section{A numerical test}
\label{sec:test}

In this section, we test the proposed modified method in Section \ref{sec:main} as well as the computational procedures in Section \ref{sec:compute} on an example. The exact solutions of this example, denoted by $u_{i}^{ex}$, are supposed to be known. Thereby, they are given by
\[
u_{1}^{ex}(x,y)=C_{1}\left(\frac{1}{7}e^{\frac{x}{a}}\cos(7x)-\frac{1}{5}e^{-(x-a)^2}\cos(5x)\right)\cos\left(\frac{3\pi y}{b}\right),
\]
\[
u_{2}^{ex}(x,y)=C_{2}x^{2}\cos\left(\frac{\pi y}{b}\right)+C_{2}x\cos\left(\frac{2\pi y}{b}\right).
\]

In this case, we consider the physical parameters:
\[
\alpha_{1}=\alpha_{2}=\delta_{11}=\delta_{22}=1,\quad 
\gamma_{1}=\gamma_{2}=-1,
\]
\[
\delta_{12}=\delta_{21}=\sigma_{11}=\sigma_{22}=0,\quad \sigma_{12}=\sigma_{21}=\frac{1}{2},
\]
and the forcing terms:
\begin{eqnarray*}
f_{1}\left(x,y\right) & = & \frac{C_{1}}{a}e^{\frac{x}{a}}\left(\left(\frac{1}{7a}-7a\right)\cos\left(7x\right)-2\sin\left(7x\right)\right)\cos\left(\frac{3\pi y}{b}\right)\\
 &  & +C_{1}e^{-\left(x-a\right)^{2}}\left(\frac{1}{5}\left(27 - 4\left(x-a\right)^{2}\right)\cos\left(5x\right)-4\left(x-a\right)\sin\left(5x\right)\right)\cos\left(\frac{3\pi y}{b}\right)-\frac{9\pi^{2}}{b^{2}}u_{1}^{ex}-\sin\left(u_{1}^{ex}\right)+\frac{1}{2}u_{2}^{ex},
\end{eqnarray*}
\[
f_{2}\left(x,y\right)=C_{2}\left(2-\frac{\pi^{2}x^{2}}{b^{2}}\right)\cos\left(\frac{\pi y}{b}\right)-C_{2}\frac{4\pi^{2}x}{b^{2}}\cos\left(\frac{2\pi y}{b}\right)-\sin\left(u_{2}^{ex}\right)+\frac{1}{2}u_{1}^{ex}.
\]

In addition, the initial conditions \eqref{eq:initial} are expressed by
\begin{equation}
u_{0}^{1}\left(y\right)=C_{1}\left(\frac{1}{7}-\frac{1}{5}e^{-a^2}\right)\cos\left(\frac{3\pi y}{b}\right),\label{eq:ini1}
\end{equation}
\begin{equation}
u_{1}^{1}\left(y\right)=C_{1}\left(\frac{1}{7a}-\frac{2a}{5}e^{-a^2}\right)\cos\left(\frac{3\pi y}{b}\right),\label{eq:ini2}
\end{equation}
\begin{equation}
u_{0}^{2}\left(y\right)=0,\quad u_{1}^{2}\left(y\right)=C_{2}\cos\left(\frac{2\pi y}{b}\right).\label{eq:ini3}
\end{equation}

We return to the issue of the influence of high frequencies. Many model problems  in the literature rely on the truncation approach, including the Cauchy problem for the Helmholtz equation  postulated by Regi\' nska and Regi\' nski \cite{RR06}, and the classical nonlinear elliptic equation taken by Tuan et al. \cite{TTKT15}. Employing such results, our computations only need to compute up to a selected cut-off constant, which reads
\begin{equation}
n \le \vartheta\ln\left(\frac{1}{\varepsilon}\right),\quad \vartheta\in (0,1),\label{eq:truncate}
\end{equation}

As an example, we choose $\vartheta = 0.5$, then tabulate in Table \ref{tab:2} several largest admissible numbers $n$ for various noise levels.
\begin{table}
\begin{centering}
\begin{tabular}{|c|c|c|c|c|c|c|c|}
\hline 
$\varepsilon$ & $10^{-2}$ & $10^{-3}$ & $10^{-4}$ & $10^{-5}$ & $10^{-6}$ & $10^{-7}$ & $10^{-8}$\tabularnewline
\hline 
$n$ & 2 & 3 & 4 & 5 & 6 & 8 & 9\tabularnewline
\hline 
\end{tabular}
\par\end{centering}
\caption{The largest admissible numbers $n$ for various amounts of noise.}
\label{tab:2}
\end{table}

At present, we are in a position to present our numerical process. In particular, we shall solve the model problem \eqref{eq:1.2}-\eqref{eq:initial} in the unit square ($a=b=1$) with $C_1=C_2=2$. At the discretization level for this problem, a grid of mesh-points $(x_j,y_l)$ is taken properly. The mesh-grid in $x$ is introduced in Subsection \ref{subsec:Picard} while  in order to utilize the adaptive Filon-type method recently commenced in Subsection \ref{subsec:filon}, we accordingly choose the equivalent mesh-width in $y$, i.e. $h_{l}^{y}=y_{l+1}-y_{l}=b/\nu$ for $l\in\left\{0,1,...,\nu\right\}$. It is worth noting that we will not solve with only  $\nu+1$ points in $y$, the needed number is more than that. It, in general, depends on the multiplicities $\left\{m_l\right\}_{0}^{\nu}$ and that we thus denote this needed set by $\left\{y_l\right\}_{0}^{L}\supset \left\{y_l\right\}_{0}^{\nu}$ without specifying the details. Additionally, the cut-off constant, denoted by $N$, will follow \eqref{eq:truncate}, which expresses
\[
N= \left\lfloor \vartheta\ln\left(\frac{1}{\varepsilon}\right)\right\rfloor.
\]

As a consequence, we use the iteration \eqref{eq:uepmild-inner11} with the starting point $u_{i}^{\varepsilon,0}=u_{i}^{\varepsilon}(x_0)\equiv u_{0}^{i,\varepsilon}$ to compute the regularized solution. In other words, we arrive at
\begin{eqnarray}
u_{i}^{\varepsilon}\left(x_{j+1},y_{l}\right) & \approx & \sum_{n=1}^{N}\left[\varphi_{n}^{i,\varepsilon}\left(x_{j+1}\right)+\int_{x_{0}}^{x_{j+1}}\kappa_{n}^{i,\varepsilon}\left(x_{j+1}-\xi\right)\left\langle f_i\left(\xi\right),\phi_{n}^{i}\right\rangle d\xi\right]\phi_{n}^{i}\left(y_{l}\right)
\nonumber
\\
 &  & -\sum_{n=1}^{N}\sum_{0\le r\le j}\int_{x_{r}}^{x_{r+1}}\kappa_{n}^{i,\varepsilon}\left(x_{j+1}-\xi\right)\left\langle \gamma_{i}\sin\left(\delta_{i1}u_{1}^{\varepsilon,r}+\delta_{i2}u_{2}^{\varepsilon,r}\right)+\sigma_{i1}u_{1}^{\varepsilon,r}+\sigma_{i2}u_{2}^{\varepsilon,r},\phi_{n}^{i}\right\rangle d\xi\phi_{n}^{i}\left(y_{l}\right),
\label{eq:asad}
\end{eqnarray}
for $0\le l\le \nu$, where we have separated the linear and nonlinear parts in \eqref{eq:uepmild-inner11}, and used the notations in Theorem \ref{thm:unique}.

The whole process of computation will be unambiguous if we do practice with the step $j=0$ as an example. Essentially, the process at $j=0$ is performed as follows:
\begin{itemize}
\item \uline{Step 1:} Set $U^{\varepsilon,0}=\left(u_{1}^{\varepsilon,0},u_{2}^{\varepsilon,0}\right)^{t}$
by a vectorial form in $y$:
\[
U^{\varepsilon,0}=\begin{bmatrix}u_{1}^{\varepsilon,0}\left(y_{0}\right) & \cdots & u_{1}^{\varepsilon,0}\left(y_{L}\right) & u_{2}^{\varepsilon,0}\left(y_{0}\right) & \cdots & u_{2}^{\varepsilon,0}\left(y_{L}\right)\end{bmatrix}^{t}\in\mathbb{R}^{2\left(L+1\right)};
\]

\item \uline{Step 2:} Compute $\mathbb{F}^{\varepsilon,0}=\bar{\gamma}\mathbb{A}^{\varepsilon,0}+\bar{\sigma}U^{\varepsilon,0}\in\mathbb{R}^{2\left(L+1\right)}$
where the block matrices $\bar{\gamma},\bar{\sigma}\in\mathbb{R}^{2\left(L+1\right)}\times\mathbb{R}^{2\left(L+1\right)}$
and the vector $\mathbb{A}^{\varepsilon,0}\in\mathbb{R}^{2\left(L+1\right)}$ are given by:
\[
\bar{\gamma}=\begin{bmatrix}\gamma_{1} & 0 & \cdots & \cdots & 0 & 0\\
0 & \ddots &  &  &  & 0\\
\vdots &  & \gamma_{1} &  &  & \vdots\\
\vdots &  &  & \gamma_{2} &  & \vdots\\
0 &  &  &  & \ddots & 0\\
0 & 0 & \cdots & \cdots & 0 & \gamma_{2}
\end{bmatrix},
\mathbb{A}^{\varepsilon,0}=\begin{bmatrix}\sin\left(\delta_{11}u_{1}^{\varepsilon,0}\left(y_{0}\right)+\delta_{12}u_{2}^{\varepsilon,0}\left(y_{0}\right)\right)\\
\vdots\\
\sin\left(\delta_{11}u_{1}^{\varepsilon,0}\left(y_{L}\right)+\delta_{12}u_{2}^{\varepsilon,0}\left(y_{L}\right)\right)\\
\sin\left(\delta_{21}u_{1}^{\varepsilon,0}\left(y_{0}\right)+\delta_{22}u_{2}^{\varepsilon,0}\left(y_{0}\right)\right)\\
\vdots\\
\sin\left(\delta_{21}u_{1}^{\varepsilon,0}\left(y_{L}\right)+\delta_{22}u_{2}^{\varepsilon,0}\left(y_{L}\right)\right)
\end{bmatrix},
\bar{\sigma}=\begin{bmatrix}\sigma_{11} & \cdots & 0 & \sigma_{12} & \cdots & 0\\
\vdots & \ddots & \vdots & \vdots & \ddots & \vdots\\
0 & \cdots & \sigma_{11} & 0 & \cdots & \sigma_{12}\\
\sigma_{21} & \cdots & 0 & \sigma_{22} & \cdots & 0\\
\vdots & \ddots & \vdots & \vdots & \ddots & \vdots\\
0 & \cdots & \sigma_{21} & 0 & \cdots & \sigma_{22}
\end{bmatrix};
\]

\item \uline{Step 3:} Using values of $\mathbb{F}^{\varepsilon,0}$,
compute the inner products $\left\langle \gamma_{i}\sin\left(\delta_{i1}u_{1}^{\varepsilon,0}+\delta_{i2}u_{2}^{\varepsilon,0}\right)+\sigma_{i1}u_{1}^{\varepsilon,0}+\sigma_{i2}u_{2}^{\varepsilon,0},\phi_{n}^{i}\right\rangle $
by the adaptive Filon-type method for $i\in\left\{ 1,2\right\} $;

\item \uline{Step 4:} Compute the remains (the linear terms) from the first summation on the right-hand side of \eqref{eq:asad}:
\[
\int_{x_{0}}^{x_{1}}\kappa_{n}^{i,\varepsilon}\left(x_{1}-\xi\right)\left\langle f_i\left(\xi\right),\phi_{n}^{i}\right\rangle d\xi,\;\mbox{and}\;\int_{x_{0}}^{x_{1}}\kappa_{n}^{i,\varepsilon}\left(x_{1}-\xi\right)d\xi,
\]
can be approximated by {an appropriate numerical integration}.

\end{itemize}

In order to illustrate the convergence, our numerical results are computed with the discrete $\ell_{2}$-norm:
\begin{eqnarray*}
E_{i}\left(x_{j}\right)&=&
\left\Vert u_{i}^{ex}\left(x_{j},\cdot\right)-u_{i}^{\varepsilon}\left(x_{j},\cdot\right)\right\Vert\\
&:=&\left(
\sum_{0\le l\le \nu}h_{l}^{y}\left|u_{i}^{ex}\left(x_{j},y_{l}\right)-u_{i}^{\varepsilon}\left(x_{j},y_{l}\right)\right|^2\right)^{1/2},
\quad i\in\left\{1,2\right\}.
\end{eqnarray*}

Concerning the mesh-width in $x$, the determination of $h_{j}^{x}$ as well as $x_{j+1}$ relies on the subprocess for building the \emph{primary} mesh, described in the flowchart below.
\begin{figure}[!h]
\centering
\parbox{10cm}{\convertMPtoPDF{fig.2.eps}{1}{1}}
\caption{The subprocess for building the mesh in $x$, conditioned by \eqref{eq:newres}.}
\label{fig:3}
\end{figure}

\begin{figure}[!h]
	\centering
	\begin{subfigure}[!h]{.5\linewidth}
		\centering
		\includegraphics[width=\textwidth]{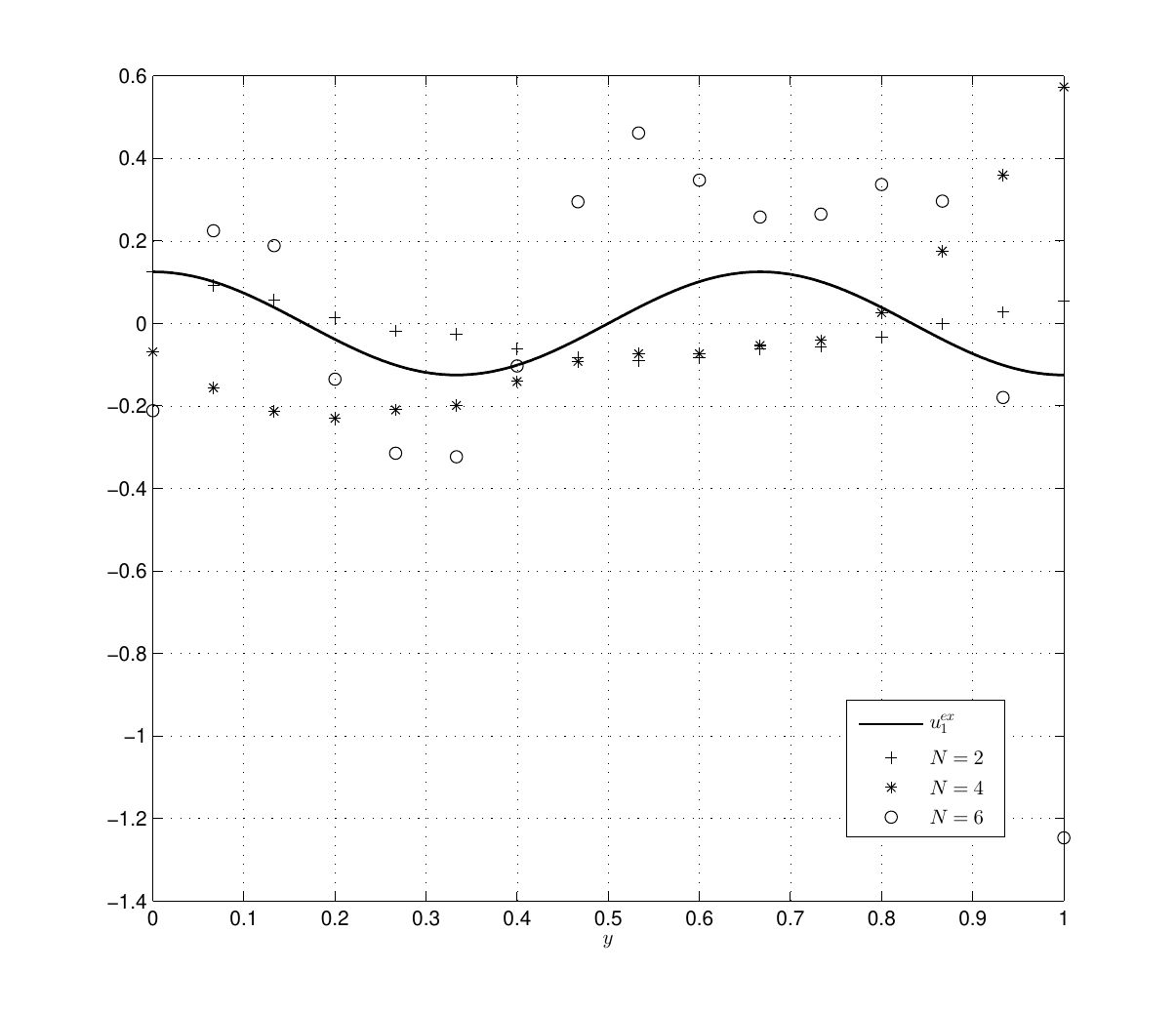}
		\subcaption{~}
	\end{subfigure}%
	\begin{subfigure}[!h]{.5\linewidth}
		\centering
		\includegraphics[width=\textwidth]{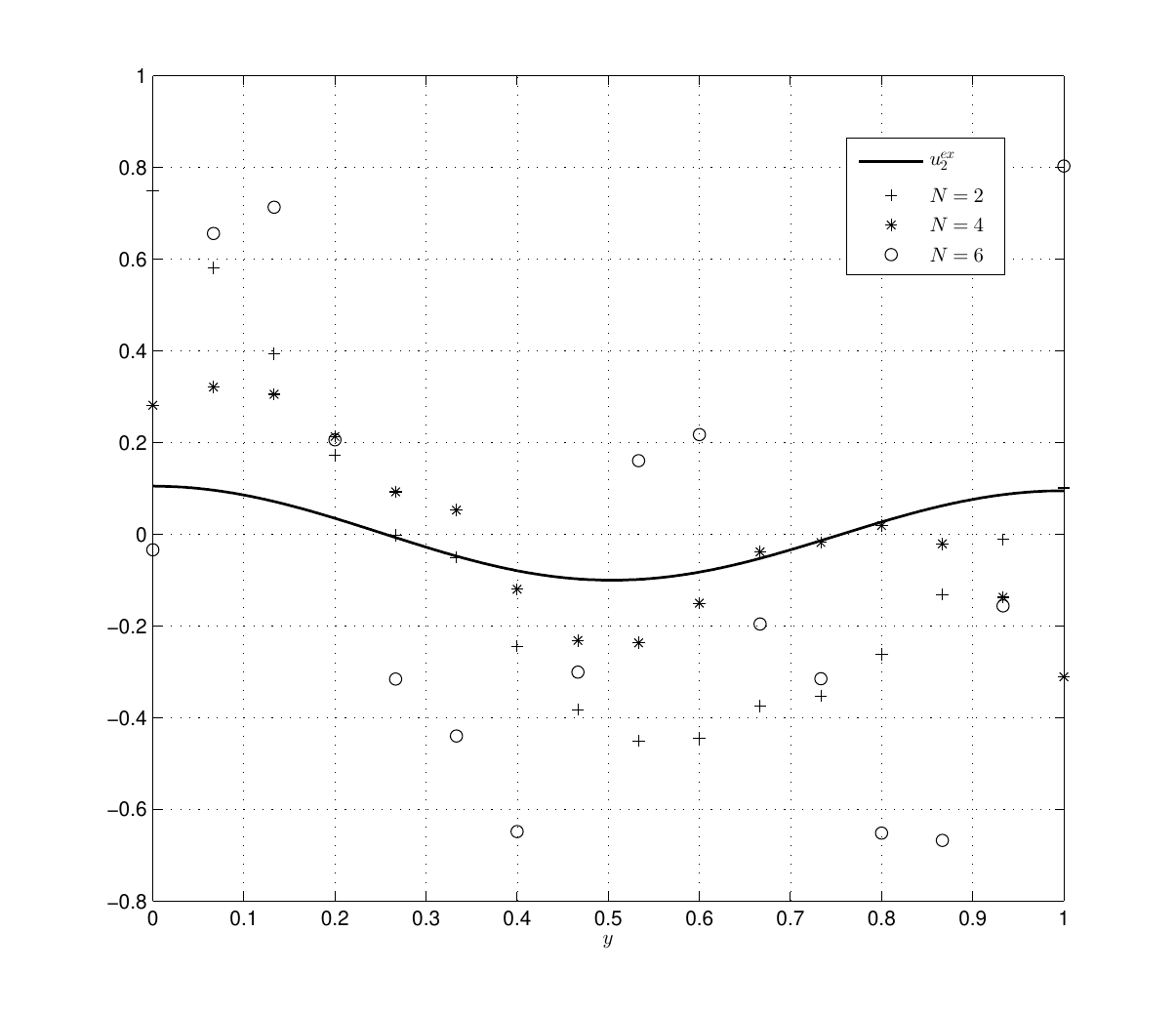}
		\subcaption{~}
	\end{subfigure}
	\caption{The unregularized solutions compared with $u_{1}^{ex}$ (left) and $u_{2}^{ex}$ (right), respectively, at $x = 0.05$ with several frequencies $N\in \left\{2,4,6\right\}$.}\label{fig:4}
\end{figure}

\begin{figure}[!h]
	\centering
	\begin{subfigure}[!h]{.5\linewidth}
		\centering
		\includegraphics[width=\textwidth]{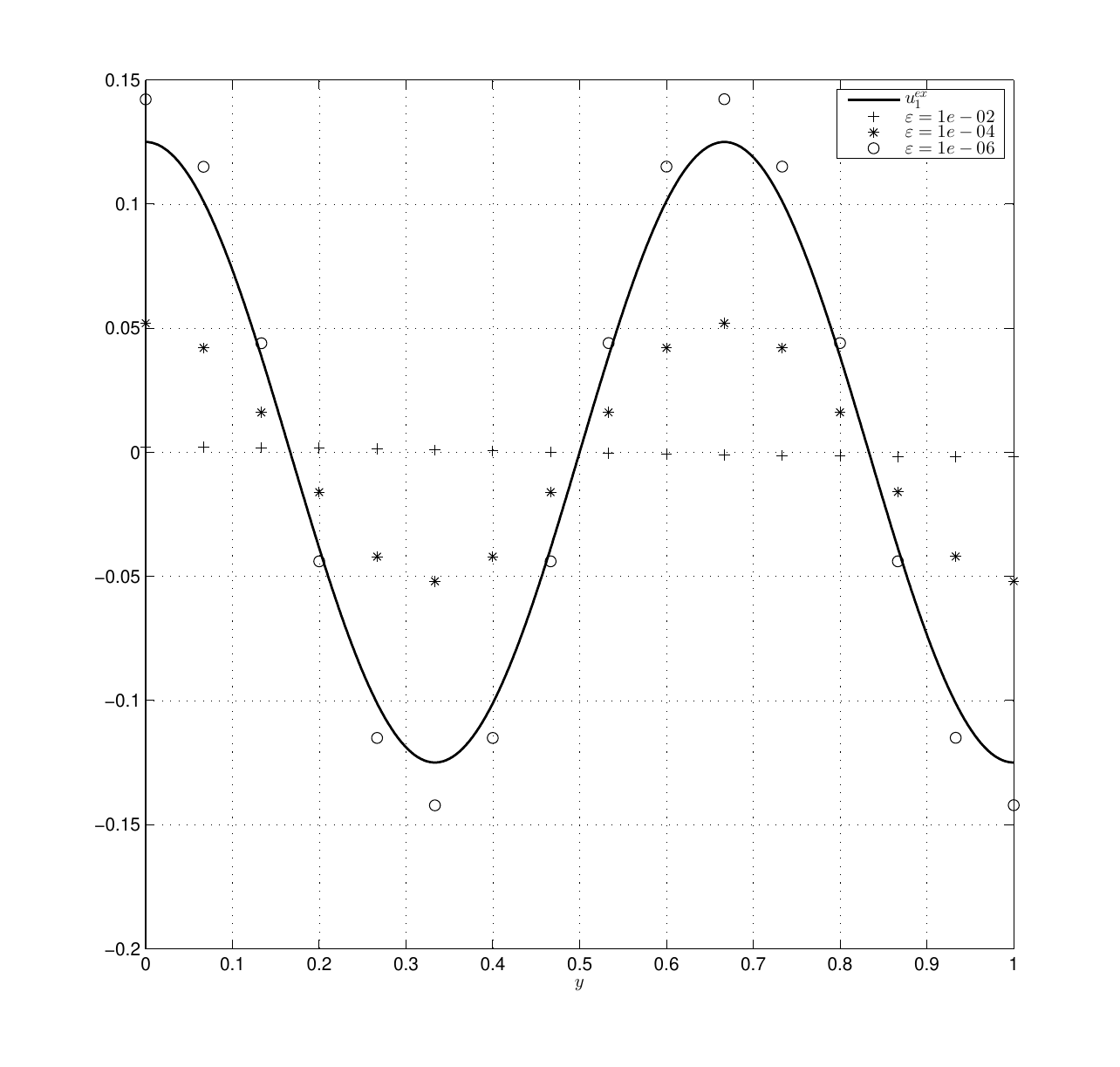}
		\subcaption{$u_{1}^{\varepsilon}$}
	\end{subfigure}%
	\begin{subfigure}[!h]{.5\linewidth}
		\centering
		\includegraphics[width=\textwidth]{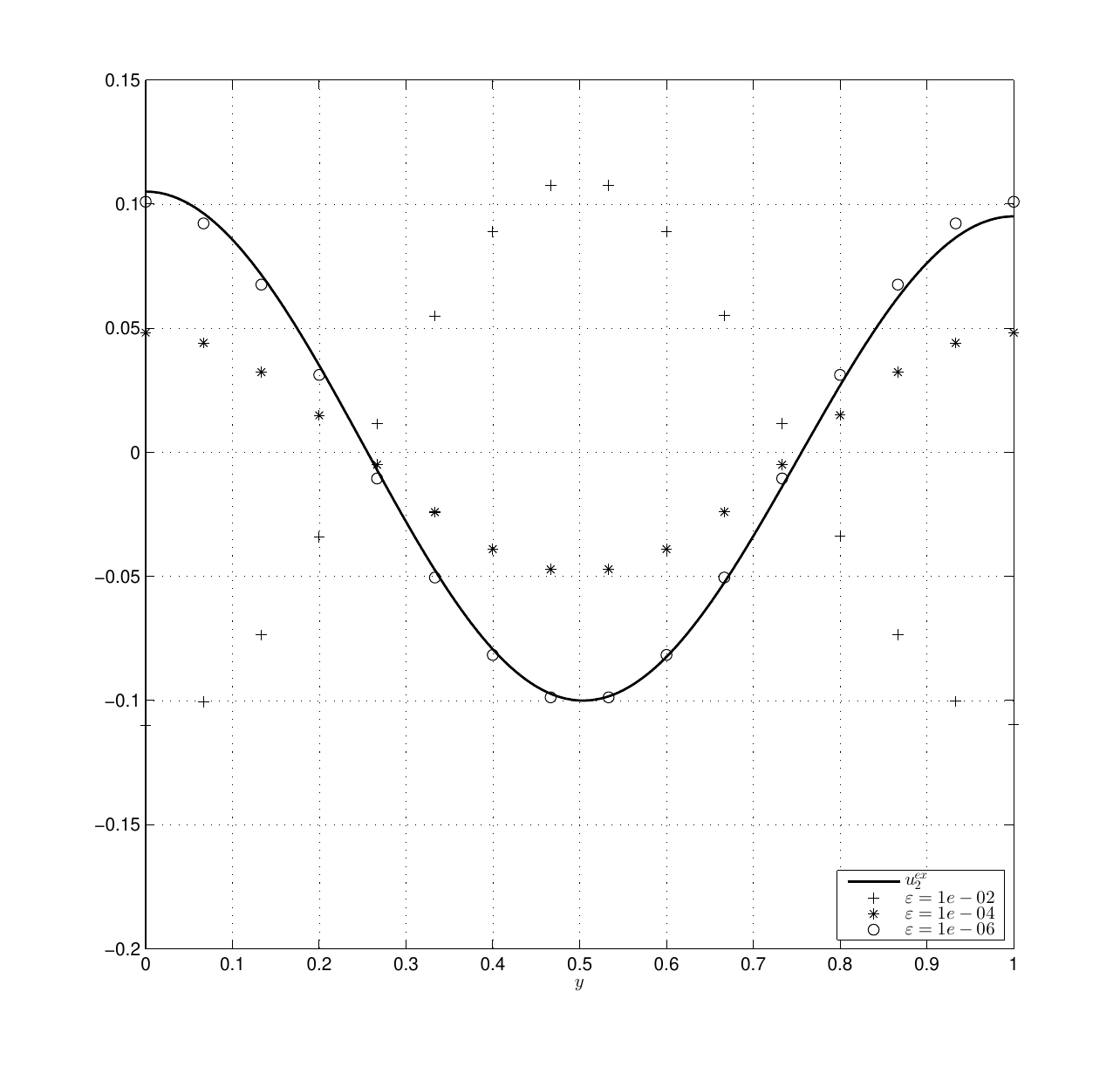}
		\subcaption{$u_{2}^{\varepsilon}$}
	\end{subfigure}
	\caption{The regularized solutions compared with $u_{1}^{ex}$ (left) and $u_{2}^{ex}$ (right), respectively, at $x=0.05$ with various amounts of noise $\varepsilon\in \left\{10^{-2},10^{-4},10^{-6}\right\}$.}\label{fig:5}
\end{figure}

\begin{figure}[!h]
	\centering
	\begin{subfigure}[!h]{.5\linewidth}
		\centering
		\includegraphics[width=\textwidth]{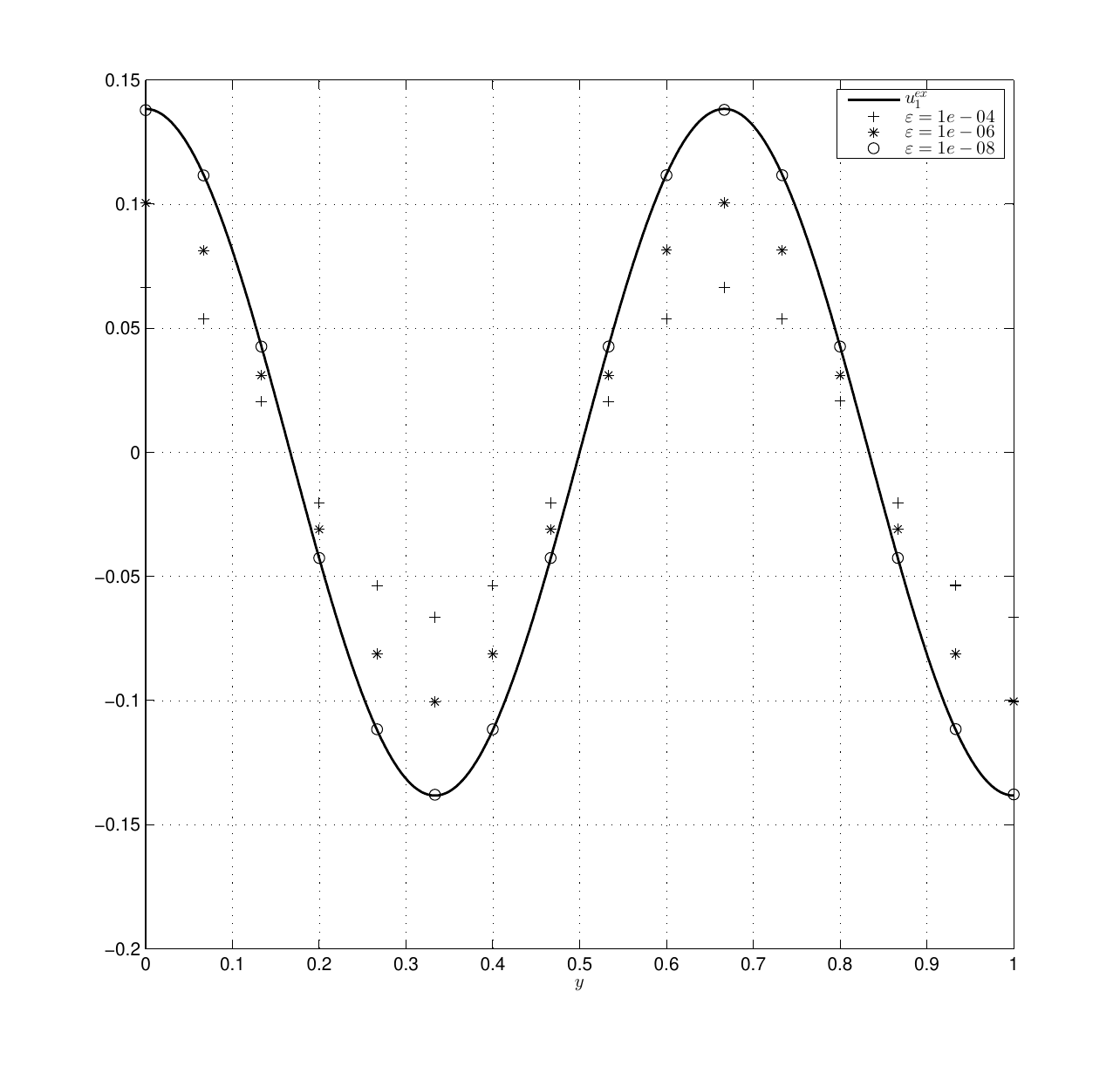}
		\subcaption{$u_{1}^{\varepsilon}$}
	\end{subfigure}%
	\begin{subfigure}[!h]{.5\linewidth}
		\centering
		\includegraphics[width=\textwidth]{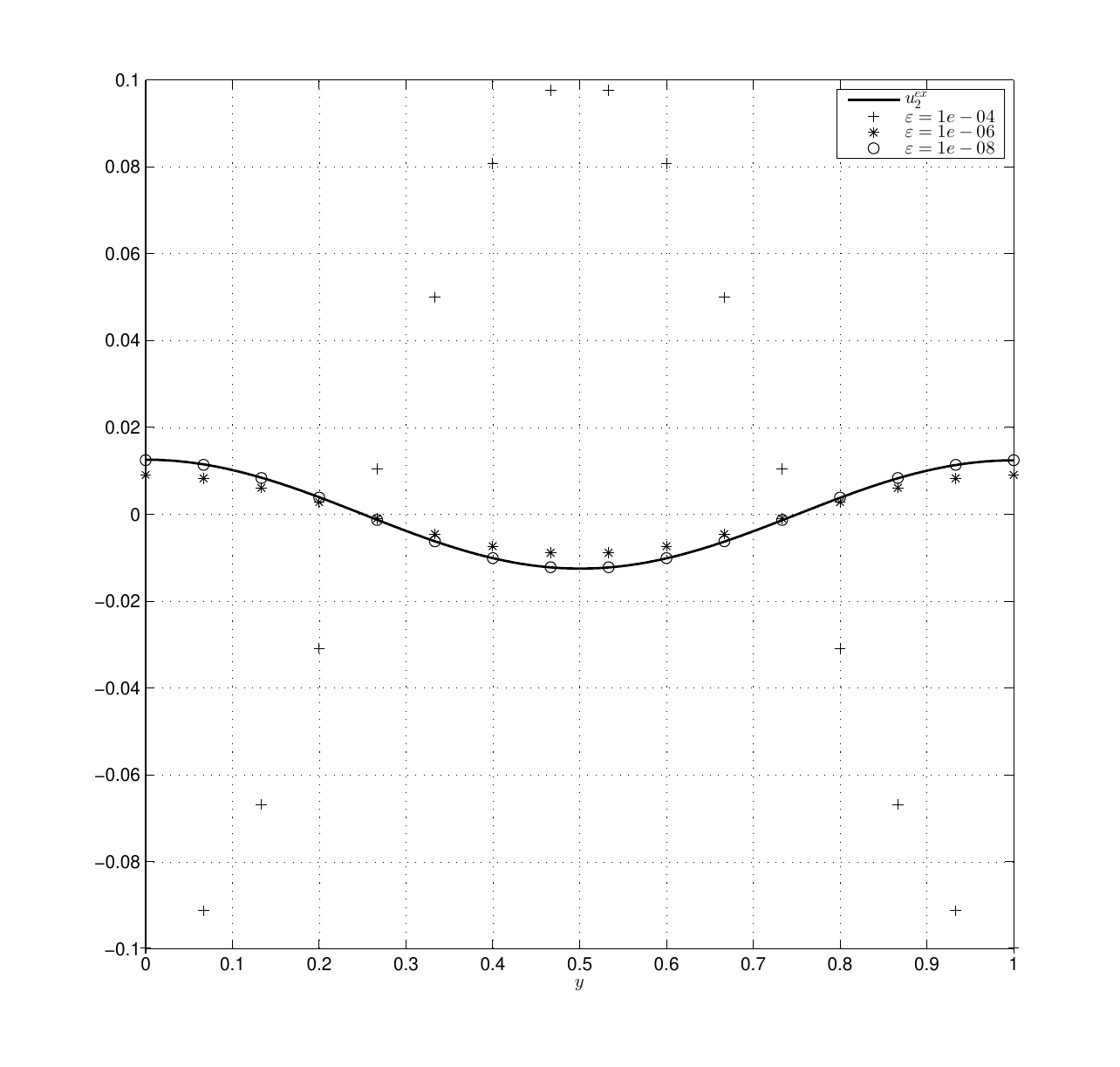}
		\subcaption{$u_{2}^{\varepsilon}$}
	\end{subfigure}
	\caption{The regularized solutions for $k=2$ compared with $u_{1}^{ex}$ (left) and $u_{2}^{ex}$ (right), respectively, at the first loop ($j=0$) with various amounts of noise $\varepsilon\in \left\{10^{-4},10^{-6},10^{-8}\right\}$.}\label{fig:6}
\end{figure}

\subsection*{Comments on numerical results}
Fixing $\nu = 15$, we begin the numerical verification by the implementation of the instability. {Note that due to the structure of forcing functions $f_{i}$, we apply the Gauss-Legendre method to approximate the integrals in Step 4 of the computational process}. In Figure \ref{fig:4}, we have plotted the unregularized solution at $x = 0.05$, computed by \eqref{eq:umild} with the exact Cauchy data. From this figure, when increasing the frequency $N$ up to 6, the obtained numerical solutions go far away the solid lines describing the corresponding exact solutions. It clearly shows the instability, albeit not being catastrophic much. It can also be seen that the distant the numerical solutions at points in $x$, the more the instability.

In the same spirit, the regularized solutions $u_{1}^{\varepsilon}$ and $u_{2}^{\varepsilon}$ for various amounts of noise $\varepsilon\in \left\{10^{-2},10^{-4},10^{-6}\right\}$ are compared in Figure \ref{fig:5} with the exact solutions at a selected point $x = 0.05$. Our intention to the selected values of noise is to compare the approximations with the above instability at the same frequencies. More precisely, we can see in Table \ref{tab:2} the corresponding admissible values of $N$ to the set of $\varepsilon$ being considered here. As introduced, we then proceed the computations with the measured Cauchy data and herein tabulated the $\ell^2$-errors in Table \ref{tab:33}. In this case, we take $k=1$, $r=0.1$ and choose the multiplicities
$m_{l}\in\left\{2,1,...,1,2 \right\}$.  This illustration agrees with our expectation that the numerical approximations become more accurate as $\varepsilon$ tends to 0.

In Figure \ref{fig:6}, we show the convergence of the numerical approximations at the first loop, i.e. $j=0$ when taking $k=2$.  Once again, this numerical result confirms our theoretical analysis and furthermore, the $\ell^2$-errors presented in Table \ref{tab:44} perform the good accuracy of the approximations. On the other side, the errors for $k=1$ are provided in  Table \ref{tab:44} to see that the error increases when $x$ goes far from the starting point 0 (compared with Table \ref{tab:33}).

Aside from the good approximations, the heavy workload of computations is viewed as the unique disadvantage. We easily get from the algorithm conditioned by \eqref{eq:newres} that the smaller the noise, the smaller the mesh $h_{j}^{x}$. For instance, from Figure \ref{fig:5} when reaching the point $x=0.05$ for just $u_{1}^{\varepsilon}$, we need to compute 8 points from $x=0$ for $\varepsilon = 10^{-2}$ whilst that number increases to 28 for $\varepsilon = 10^{-4}$, and then 126 for $\varepsilon = 10^{-6}$. However, we remark that such a requirement also decides to the convergence of the approximations, i.e. choosing small $r$ gives the more accuracy. Essentially, this trade-off is unavoidable and acceptable. 

\begin{table}
\begin{centering}
\begin{tabular}{|c|c|c|c|}
\hline 
 & $\varepsilon=10^{-2}$ & $\varepsilon=10^{-4}$ & $\varepsilon=10^{-6}$\tabularnewline
\hline 
$E_{1}$ & 9.3914$\times10^{-2}$ & 5.4648$\times10^{-2}$ & 1.2978$\times10^{-2}$\tabularnewline
\hline 
$E_{2}$ & 1.5809$\times10^{-1}$ & 4.0363$\times10^{-2}$ & 3.8271$\times10^{-3}$\tabularnewline
\hline 
\end{tabular}
\par\end{centering}

\caption{The numerical errors at $x=0.05$ for various amounts of noise $\varepsilon\in\left\{ 10^{-2},10^{-4},10^{-6}\right\} $.}\label{tab:33}
\end{table}

\begin{table}
\begin{centering}
	\begin{subtable}{.5\linewidth}
		\centering
		\begin{tabular}{|c|c|c|c|}
			\hline 
			& $\varepsilon=10^{-4}$ & $\varepsilon=10^{-6}$ & $\varepsilon=10^{-8}$\tabularnewline
			\hline 
			$E_{1}$ & 5.0302$\times10^{-2}$  & 5.4452$\times10^{-3}$  & 2.3830$\times10^{-4}$ \tabularnewline
			\hline 
			$E_{2}$ & 3.1348$\times10^{-2}$  & 4.1936$\times10^{-4}$  & 5.8834$\times10^{-5}$ \tabularnewline
			\hline 
		\end{tabular}
		\caption{$k=1$}
	\end{subtable}
	\begin{subtable}{.5\linewidth}
		\centering
		\begin{tabular}{|c|c|c|c|}
			\hline 
			& $\varepsilon=10^{-4}$ & $\varepsilon=10^{-6}$ & $\varepsilon=10^{-8}$\tabularnewline
			\hline 
			$E_{1}$ & 5.4141$\times10^{-2}$  & 2.8486$\times10^{-2}$  & 3.0083$\times10^{-4}$ \tabularnewline
			\hline 
			$E_{2}$ & 8.4573$\times10^{-2}$  & 2.5777$\times10^{-3}$  & 6.3485$\times10^{-5}$\tabularnewline
			\hline 
		\end{tabular}
		\caption{$k=2$}
	\end{subtable}
\end{centering}

\caption{The numerical errors at the first loop $j=0$ $\left(x=0.00625\right)$
for $k\in\left\{ 1,2\right\} $ and various amounts of noise $\varepsilon\in\left\{ 10^{-4},10^{-6},10^{-8}\right\} $.}
\label{tab:44}
\end{table}

\section{Concluding remarks}
\label{sec:conclusion}

We have opened a new stage on the regularization for the coupled elliptic sine-Gordon equations along with Cauchy data. The precedent versions where single equations, including \cite{TTK15}, are investigated can be found in the references. In the current paper, we have successfully extended the general kernel-based regularization method, commenced in \cite{TTTK15} and proposed in \eqref{eq:uepmild}, wherein the Gevrey-type regularity plays
a strongly powerful position in our theoretical analysis. We have demonstrated in Theorem \ref{thm:unique} and Theorem \ref{thm:3} the well-posedness of the regularized solution, and the centrepiece of this study was the derivation of the convergence rate in Theorem \ref{thm:5}. We have also ventured far beyond the existing investigation by rigorously focusing on the numerical procedures for the regularized solution, which were not well-treated in the past. This framework therein encompasses the use of the asymptotic expansions and the Filon-type methods in combination with the Picard-based iteration, which were carefully postulated in \cite{IN04,Olv08}, within the setting of discretization in two-dimensional. Due to the choice of the cut-off constant we primarily concentrated on the Filon-type methods.

Aside from the pioneering works in the field of regularization methods, we, from the numerical point of view, pointed out how big the cut-off constant is, which is more or less alike the statement by Zhang et al. \cite{ZW14}. On the other side, the simple algorithm searching for the suitable noise-dependent mesh-grid in $x$ was conditioned and proposed in Figure \ref{fig:3} whilst the mesh-grid in $y$ for the usage of the adaptive Filon-type method was illustrated in Figure \ref{fig:1}. The numerical implementation was implemented using MATLAB and the computations were done on a computer equipped with processor Intel Core i5-3330 4x3.20GHz and having 4.0 GB of RAM.

Our paper can be adapted to handle more complex scenarios. For instance, the parameters $\gamma_{i}$ can be more realistic if they depend on spatial variables, and they are in this sense viewed as the Josephson current densities. Furthermore, we can deal with the model problem with a large amount of coupled equations as well as high-dimension problems. However, we note that the more the equations and dimensions, the more complexity of computations we shall face.

\section*{Acknowledgment}
This work was initiated when V.A.K was visiting Faculty of Mathematics and Statistics, Ton Duc Thang University in Vietnam. The main results were completed when V.A.K enjoyed the hospitality of Department of Mathematics and Computer Science, Karlstad University in Sweden. V.A.K also acknowledges Prof. Ernst Hairer for fruitful discussions about the highly oscillatory integrals since his teaching tasks at GSSI in 2015. N.H.T  gratefully acknowledges stimulating discussions with
Prof. Daniel Lesnic.

\bibliography{mybib}

\begin{thebibliography}{10}
\expandafter\ifx\csname url\endcsname\relax
  \def\url#1{\texttt{#1}}\fi
\expandafter\ifx\csname urlprefix\endcsname\relax\def\urlprefix{URL }\fi
\expandafter\ifx\csname href\endcsname\relax
  \def\href#1#2{#2} \def\path#1{#1}\fi

\bibitem{Chen03}
G.~Chen, Z.~Ding, C.-R. Hu, W.-M. Ni, J.~Zhou, A note on the elliptic
  sine-{G}ordon equation, Contemporary Mathematics 357 (2004) 49--68.

\bibitem{Le84}
M.~Levi, Beating modes in the {J}osephson junction, in {C}haos in {N}onlinear
  {D}ynamical {S}ystems, SIAM, Philadelphia, 1984.

\bibitem{HN02}
J.-H. Ha, S.~Nakagiri, Identification problems of damped sine-{G}ordon
  equations with constant parameters, Journal of the Korean Mathematical
  Society 39~(4) (2002) 509--524.

\bibitem{BK98}
O.~M. Braun, Y.~S. Kivshar, Nonlinear dynamics of the {F}renkel--{K}ontorova
  model, Physics Reports 306~(1) (1998) 1--108.

\bibitem{Yo83}
S.~Yomosa, Soliton excitations in deoxyribonucleic acid ({DNA}) double helices,
  Physical Review A 27~(4) (1983) 2120--2125.

\bibitem{Cap98}
J.-G. Caputo, N.~Flytzanis, Y.~Gaididei, I.~Moulitsa, E.~Vavalis, Split mode
  method for the elliptic 2{D} sine-{G}ordon equation: application to
  {J}osephson junction in overlap geometry, International Journal of Modern
  Physics C 9~(02) (1998) 301--323.

\bibitem{NS97}
V.~Y. Novokshenov, A.~G. Shagalov, Bound states of the elliptic sine-{G}ordon
  equation, Physica D: Nonlinear Phenomena 106~(1) (1997) 81--94.

\bibitem{Sha92}
A.~G. Shagalov, Singular solutions of the elliptic sine-{G}ordon equation:
  models of defects, Physics Letters A 165~(5) (1992) 412--416.

\bibitem{Ray06}
S.~S. Ray, A numerical solution of the coupled sine-{G}ordon equation using the
  modified decomposition method, Applied mathematics and computation 175~(2)
  (2006) 1046--1054.

\bibitem{ARRV09}
G.~Alessandrini, L.~Rondi, E.~Rosset, S.~Vessella, The stability for the
  {C}auchy problem for elliptic equations, Inverse Problems 25~(12) (2009)
  123004, 47 pp.

\bibitem{BD10}
L.~Bourgeois, J.~Dard, About stability and regularization of ill-posed elliptic
  {C}auchy problems: the case of {L}ipschitz domains, Applicable Analysis
  89~(11) (2010) 1745--1768.

\bibitem{Bour07}
L.~Bourgeois, A stability estimate for ill-posed elliptic {C}auchy problems in
  a domain with corners, Comptes Rendus Mathematique 345~(7) (2007) 385--390.

\bibitem{HDL09}
D.~N. H\`ao, N.~V. Duc, D.~Lesnic, A non-local boundary value problem method
  for the {C}auchy problem for elliptic equations, Inverse Problems 25~(5)
  (2009) 055002.

\bibitem{HL00}
D.~N. H\`ao, D.~Lesnic, The {C}auchy problem for {L}aplace's equation via the
  conjugate gradient method, IMA Journal of Applied Mathematics 65~(2) (2000)
  199--217.

\bibitem{RT09}
T.~Reginska, U.~Tautenhahn, Conditional stability estimates and regularization
  with applications to {C}auchy problems for the {H}elmholtz equation,
  Numerical Functional Analysis and Optimization 30~(9-10) (2009) 1065--1097.

\bibitem{TTTK15}
N.~H. Tuan, L.~D. Thang, D.~D. Trong, V.~A. Khoa, Approximation of mild
  solutions of the linear and nonlinear elliptic equation, Inverse Problems in
  Science and Engineering 23~(7) (2015) 1237--1266.

\bibitem{TTKT15}
N.~H. Tuan, L.~D. Thang, V.~A. Khoa, T.~Tran, On an inverse boundary value
  problem of a nonlinear elliptic equation in three dimensions, Journal of
  Mathematical Analysis and Applications 426~(2) (2015) 1232--1261.

\bibitem{Olv08}
S.~Olver, Numerical approximation of highly oscillatory integrals, Ph.D.
  thesis, University of Cambridge (2008).

\bibitem{IN04}
A.~Iserles, S.~N\o{}rsett, On quadrature methods for highly oscillatory
  integrals and their implementation, BIT Numerical Mathematics 44~(4) (2004)
  755--772.

\bibitem{IN05}
A.~Iserles, S.~N\o{}rsett, Efficient quadrature of highly oscillatory integrals
  using derivatives, Proceedings of the Royal Society A 461 (2005) 1383--1399.

\bibitem{TTK15}
N.~H. Tuan, L.~D. Thang, V.~A. Khoa, A modified integral equation method of the
  nonlinear elliptic equation with globally and locally {L}ipschitz source,
  Applied Mathematics and Computation 265 (2015) 245--265.

\bibitem{LL67}
R.~Latt\`es, J.~Lions, Methode de quasi-reversibilite et applications, Dunod,
  Collier-Macmillan Ltd., Paris, 1968.

\bibitem{Filon}
L.~Filon, On a quadrature formula for trigonometric integrals, Proc. Roy. Soc.
  Edinburgh Sect. A 49~(1) (1928) 38--47.

\bibitem{Stein93}
E.~M. Stein, Harmonic analysis: real-variable methods, orthogonality, and
  oscillatory integrals, Princeton Mathematical Series 43.

\bibitem{RR06}
T.~Regi\'{n}ska, K.~Regi\'{n}ski, Approximate solution of a {C}auchy problem
  for the {H}elmholtz equation, Inverse Problems 22 (2006) 975--989.

\bibitem{ZW14}
H.~Zhang, T.~Wei, A {F}ourier truncated regularization method for a {C}auchy
  problem of a semi-linear elliptic equation, Journal of Inverse and Ill-posed
  Problems 22~(2) (2014) 143--168.

\end{thebibliography}

\end{document}